\documentclass[12pt]{article}
\usepackage[]{amsmath,amssymb}
\usepackage{amscd}
\usepackage{latexsym}
\usepackage{cite}
\usepackage{amsthm}

\newtheorem{definition}{Definition}[section]
\newtheorem{theorem}[definition]{Theorem}
\newtheorem{lemma}[definition]{Lemma}

\newtheorem{remark}[definition]{Remark}

\newtheorem{problem}[definition]{Problem}
\newtheorem{note}[definition]{Note}

\newtheorem{proposition}[definition]{Proposition}

\begin{document}
\title{\bf
The Attenuated Space Poset $\mathcal{A}_q(N, M)$
}
\author{
Wen Liu}
\date{}

\maketitle
\begin{abstract}
In this paper, we
 study the incidence algebra $T$ of the
attenuated space poset $\mathcal{A}_q(N, M)$.
We consider the following topics.
We consider some generators of $T$: the raising matrix $R$, the lowering matrix $L$, and a certain diagonal matrix $K$.
  We describe some relations among $R, L, K$. We put these relations in an attractive form using a certain matrix $S$
in $T$. We characterize
the center  $Z(T)$. Using  $Z(T)$, we relate $T$ to the quantum group  $U_\tau({\mathfrak{sl}}_2)$ with $\tau^2=q$.
We consider two elements $A, A^*$ in $T$ of a certain form. We find
necessary and sufficient conditions for $A, A^*$ to satisfy the tridiagonal relations.
Let $W$ denote an irreducible $T$-module. We find necessary and sufficient conditions for the above $A, A^*$
to act on $W$ as a Leonard pair.

 \end{abstract}
\section{Introduction}
In \cite{ter5}, Terwilliger introduced the incidence algebra of  a uniform poset.  This algebra is motivated by the
Terwilliger algebra of a  $Q$-polynomial distance-regular graph. For many $Q$-polynomial distance-regular graphs, the Terwilliger algebra
  is related to a quantum group \cite{gaogaohou, gaozhanghou, itoter1, itoter2, tomiyama, woraw}.
So  for a uniform poset, it is natural to ask whether its  incidence algebra is related to a quantum group.
In this paper we will show that this is the case for  certain uniform posets.

We recall some examples of uniform posets.
In  \cite{ter5}, Terwilliger used the classical geometries to obtain eleven families of uniform posets.
The polar spaces give one of the families \cite[Example 3.1]{ter5}.
In \cite{woraw},   Worawannotai found
 another family of uniform posets using the polar spaces.
For each bipartite $Q$-polynomial distance-regular graph,
Miklavi$\check{c}$ and Terwilliger~\cite{miklavic} considered
 a poset on its vertex set. They found necessary and sufficient conditions for this poset to be  uniform.
In \cite{nakang},  Kang and Chen  obtained a family of uniform posets using the nonisotropic subspaces of  a  unitary polar space.

The incidence algebra $T$ of a uniform poset is finite-dimensional and semisimple \cite{ter5}.
In \cite{ter5}, Terwilliger  gave a method for computing the irreducible $T$-modules.
To describe these modules, it is convenient to use the notion of a Leonard pair.
This notion was  introduced by Terwilliger~\cite{ter1, ter2,  ter3} as an abstraction of some work of
Leonard \cite{leonard} concerning the orthogonal polynomials in the terminating branch of the
Askey scheme \cite{koekoek}.  Leonard pairs are closely related to quantum groups \cite{alnajjar, ter20151231, ter7+, ter9, ter8, ter16041, ter16042, tervid}.
Leonard pairs are also related to  $Q$-polynomial distance-regular graphs  \cite{lee, terzit, woraw}.  We mentioned a similarity
between the Terwilliger algebra for a $Q$-polynomial distance-regular graph and the incidence algebra for a uniform poset.
Given a uniform poset, it is natural to look for a Leonard pair structure on each irreducible $T$-module. These
Leonard pairs were found for certain examples \cite{nakang, ter6,  ter20151231}.

There is a family of classical geometries called the attenuated spaces.
An attenuated space admits the structure of a uniform poset  \cite{ter5}. Bonoli and Melone~\cite{Bonoli} gave a geometrical characterization of an attenuated space. Wang, Guo, Li \cite{wgl1, guoliwang} constructed an association scheme based on  an attenuated space. They computed all its intersection numbers and studied its incidence matrices. Kurihara
\cite{Kurihara} computed the character table of this association scheme.  Gao and Wang \cite{gaowang}
constructed some error-correcting codes based on an attenuated space.
 Liu and Wang \cite{LW1}
characterized the full automorphism group of some  graphs based on  an attenuated space.

An attenuated space gives a uniform poset called $\mathcal{A}_q(N, M)$.
In this paper we
 study the incidence algebra $T$ of $\mathcal{A}_q(N, M)$. We consider the following topics.
  We consider some generators of $T$: the raising matrix $R$, the lowering matrix $L$, and a certain diagonal matrix $K$.
  We describe some relations among $R, L, K$. We put these relations in an attractive form using a certain matrix $S$
in $T$. We characterize
the center  $Z(T)$. Using  $Z(T)$, we relate $T$ to the quantum group  $U_\tau({\mathfrak{sl}}_2)$ with $\tau^2=q$.
We consider two elements $A, A^*$ in $T$ of a certain form. We find
necessary and sufficient conditions for $A, A^*$ to satisfy the tridiagonal relations \cite{itotanater}.
Let $W$ denote an irreducible $T$-module. We find necessary and sufficient conditions for the above $A, A^*$
to act on $W$ as a Leonard pair.

The paper is organized  as follows:
In Section 2, we recall some basic definitions and facts about the attenuated space poset
 $\mathcal{A}_q(N,M)$ and
its incidence algebra $T$.
In Section 3
 we display some relations among $R, L, K$. In Section 4 we introduce the matrix $S$ and use it to simplify these relations.
In Section 5, we recall some facts from \cite{ter5} about the irreducible
$T$-modules and the center $Z(T)$.
In Sections 6, 7
we display two families of central elements  of $T$. We show that each family  generates $Z(T)$.
In Section 8 we relate $T$ to the quantum group  $U_\tau({\mathfrak{sl}}_2)$ with $\tau^2=q$.
In Section 9, we obtain some results about $R, L$ that will be used in Section 11.
In Section 10,
we recall some definitions and facts about Leonard pairs.
In Section 11 we consider   two elements $A, A^*$ in $T$ of a certain form. We find
necessary and sufficient conditions for $A, A^*$ to satisfy the tridiagonal relations.
In Section 12 we consider the actions of $A, A^*$ on the  irreducible $T$-modules.


\section{The  attenuated space poset and its incidence algebra}
In this section, we first recall some basic definitions
 and  facts about the
   attenuated space poset $\mathcal{A}_{q}(N,M)$.
We then recall    the incidence algebra $T$ of $\mathcal{A}_{q}(N,M)$.
This material
is mainly taken from \cite{ter6}.

Throughout this section fix positive integers
$N, M$. Fix a finite field $\mathbb{F}_{q}$ of order $q$.
We will be discussing the square root of $q$. Throughout the paper,
fix a square root $q^{1/2}$.
Let $H$ be a vector space over  $\mathbb{F}_{q}$ that has dimension $N+M$.
    Fix an $M$-dimensional subspace $h$  of $H$.
 Let $P$ denote the set of  subspaces
of $H$ whose intersection with $h$ is zero.
For $x,y\in P$ write  $x\leq y$ whenever $x\subseteq y$.
This relation $\leq$  is a partial order on $P$.
The poset $P$ is called the {\em attenuated space
poset} and often denoted by $\mathcal{A}_{q}(N,M)$.

 For $x,y\in P$ write $x<y$ whenever $x\leq y$ and $x\neq y$.
We say that $y$ {\em covers} $x$ whenever $x<y$
 and there is no $z\in P$ such that $x<z<y$.
For $0\leq i\leq N$, let $P_i$ denote the set of
 elements in $P$ that have dimension $i$.
The sequence $\{P_i\}^N_{i=0}$ is a grading of $P$
in the sense of \cite{ter5}.
 Let $\mathbb{C}P$ denote the vector space over
$\mathbb{C}$ consisting of all formal $\mathbb{C}$-linear
 combinations of elements in $P$. The set $P$
  is a basis for $\mathbb{C}P$,
  so the dimension of $\mathbb{C}P$ is equal to
   the cardinality of $P$.


The {\em lowering matrix} $L\in {\rm Mat}_P(\mathbb{C})$ and
 the {\em raising matrix} $R\in {\rm Mat}_P(\mathbb{C})$ have entries
\begin{equation}\label{lr}
L_{xy}= \left\{\begin{array} {cl}
1, & {\rm if}~y~{\rm covers}~x;\\
0, & {\rm if}~y~{\rm does~ not~cover}~x
\end{array}\right. \qquad\qquad
R_{xy}= \left\{\begin{array} {cl}
1, & {\rm if}~x~{\rm covers}~y;\\
0, & {\rm if}~x~{\rm does~ not~ cover}~y
\end{array}\right.
\end{equation}
for $x, y\in P$. Note that the transpose
 $R^t=L$.
For $0\leq i\leq N$ let $F_i$ denote
 the diagonal matrix in ${\rm Mat}_P(\mathbb{C})$ with  diagonal entries
\begin{equation}\label{fi}
(F_i)_{yy}= \left\{\begin{array} {cc}
1, & {\rm if}~~y\in P_i;\\
0, & {\rm if}~~y\notin P_i
\end{array}\right.
\qquad\qquad y\in P.
\end{equation}
We have
\begin{equation}\label{12151}
F_iF_j=\delta_{ij}F_i\qquad\qquad\qquad (0\leq i,j\leq N)
\end{equation}
and
\begin{equation}\label{12152}
I=\sum\limits_{i=0}^NF_i.
\end{equation}
We refer to $F_i$ as the  {\em $i$th projection matrix} for $P$. The
$\{F_i\}^N_{i=0}$ are related to the $R, L$ by
\begin{eqnarray}
& &RF_i=F_{i+1}R\qquad (0\leq i\leq N-1), \qquad  RF_N=0,\qquad F_0R=0,\label{12153}\\
& &LF_i=F_{i-1}L\qquad (1\leq i\leq N), \qquad\qquad  LF_0=0,\qquad ~F_NL=0. \label{12154}
\end{eqnarray}

Define
\begin{equation}\label{k}
K=\sum\limits_{i=0}^{N}q^{N+M-i}F_i.
\end{equation}
Note that $K$ is diagonal. Moreover, $K$ is invertible and
\begin{equation}\label{kinverse}
K^{-1}=\sum\limits_{i=0}^{N}q^{i-N-M}F_i.
\end{equation}
The {\em incidence algebra}  $T$ of $P$ is  the
 subalgebra of ${\rm Mat}_P(\mathbb{C})$
generated by $R, L, K^{\pm 1}$. The vector space $\mathbb{C}P$
is a $T$-module. By \cite[Theorem~2.5]{ter5} the algebra $T$
is semisimple.
Therefore  the $T$-module  $\mathbb{C}P$ is a direct sum of
irreducible $T$-submodules.

We mention two basic facts for later use.
\begin{lemma}\label{lem:firkfj}
For $0\leq i,j,k\leq N$,
\begin{eqnarray*}
& &F_iR^kF_j\neq 0\qquad\qquad{\rm iff}\qquad\qquad k=i-j,\label{firkf1}
\\
& &F_iL^kF_j\neq 0\qquad\qquad{\rm iff}\qquad\qquad k=j-i.\label{firkf2}
\end{eqnarray*}
\end{lemma}

\begin{lemma}\label{lem:fbf}
For  $D\in {\rm Mat}_P(\mathbb{C})$ the following are equivalent:
\begin{enumerate}
\item[\rm(i)] $D=0$;
\item[\rm(ii)]  $F_iDF_j=0$ for $0\leq i, j\leq N$.
\end{enumerate}
\end{lemma}
\begin{proof} ${\rm(i)}\Rightarrow  {\rm(ii)}$  Clear.

\noindent
${\rm(ii)}\Rightarrow  {\rm(i)}$ By (\ref{12152}),
$$
D=IDI=\sum\limits_{i=0}^N\sum\limits_{j=0}^N  F_iDF_j.
$$
The result follows.
\end{proof}

\section{How  $R,L,K$ are related}
In this section, we describe some relations among $R, L, K$.

\begin{lemma}\label{lem:relationrl}
The matrices $R, L, K$ satisfy the following relations:
\begin{equation}\label{relation:1}
RK=qKR,\qquad \qquad LK=q^{-1}KL,
\end{equation}
\begin{equation}\label{relation:2}
q(q+1)^{-1}RL^2-LRL+(q+1)^{-1}L^2R+LK=0,
\end{equation}
\begin{equation}\label{relation:3}
q(q+1)^{-1}R^2L-RLR+(q+1)^{-1}LR^2+KR=0.
\end{equation}
\end{lemma}
\begin{proof} To obtain (\ref{relation:1}), in each equation compare the
$(x,y)$-entries of each side for $x,y\in P$.
The equation (\ref{relation:2}) is obtained by \cite[Theorem~3.2]{ter5}. The equation (\ref{relation:3}) is obtained from (\ref{relation:2})  by applying the transpose to
each side.
\end{proof}

\begin{lemma}\label{lem:1603031}
We have
\begin{eqnarray}\label{201603031}
R^2LR=q^{-2}(q+1)R^2K+\frac{q(q^2-1)R^3L+(q-1)LR^3}{q^3-1},
\end{eqnarray}
\begin{eqnarray}\label{201603032}
RLR^2=q^{-2}(q+1)R^2K+\frac{q^2(q-1)R^3L+(q^2-1)LR^3}{q^3-1}.
\end{eqnarray}
Moreover,
\begin{eqnarray}\label{201603034}
L^2RL=(q+1)L^2K+\frac{q^2(q-1)RL^3+(q^2-1)L^3R}{q^3-1},
\end{eqnarray}
\begin{eqnarray}\label{201603033}
LRL^2=(q+1)L^2K+\frac{q(q^2-1)RL^3+(q-1)L^3R}{q^3-1}.
\end{eqnarray}
\end{lemma}
\begin{proof}
In equation (\ref{relation:3}) multiply each term on the left by $R$
and simplify the result using (\ref{relation:1}) to obtain
\begin{eqnarray}
q(q+1)^{-1}R^3L-R^2LR+(q+1)^{-1}RLR^2+qKR^2=0.\label{relation:5}
\end{eqnarray}
In equation (\ref{relation:3}) multiply each term on the right by $R$ to obtain
\begin{eqnarray}
q(q+1)^{-1}R^2LR-RLR^2+(q+1)^{-1}LR^3+KR^2=0.\label{relation:6}
\end{eqnarray}
Combining (\ref{relation:5}), (\ref{relation:6}) we obtain (\ref{201603031}), (\ref{201603032}). Apply the transpose map to
(\ref{201603031}), (\ref{201603032}) to get (\ref{201603034}), (\ref{201603033}).
\end{proof}
For a nonzero $\tau\in \mathbb{C}$ such that $\tau^2\neq 1$, define
$$
[n]_\tau=\frac{\tau^n-\tau^{-n}}{\tau-\tau^{-1}},\qquad\qquad n=0,1,2,\ldots
$$
\begin{lemma}\label{lem:201601061}
The matrices $R, L$ satisfy the cubic $q^{1/2}$-Serre relations:
\begin{eqnarray*}
R^3L-[3]_{q^{1/2}}R^2LR+[3]_{q^{1/2}}RLR^2-LR^3=0,\label{relation:8++}\\
L^3R-[3]_{q^{1/2}}L^2RL+[3]_{q^{1/2}}LRL^2-RL^3=0. \label{relation:8+++}
\end{eqnarray*}
\end{lemma}
\begin{proof}
By Lemma~\ref{lem:1603031}.
\end{proof}
\begin{lemma}\label{lem:rlcommute}
The matrices $RL, LR, K, K^{-1}$  mutually commute.
\end{lemma}
\begin{proof}
By (\ref{relation:1}), each of $RL, LR$ commutes with $K$. To see that $RL, LR$ commute, use (\ref{relation:1})--(\ref{relation:3}).
 \end{proof}

\section{The matrix $S$}
In this section, we define a  matrix $S\in{\rm Mat}_P(\mathbb{C})$ and we discuss how
$S$ is related to $R,L,K$.

\begin{definition}\label{def:piti}
For  $1\leq i\leq N+1$, denote
by $\tilde{P}_i$ the set of all
$i$-dimensional subspaces of $H$ whose intersection with $h$ has
dimension one.
\end{definition}

Define $S\in {\rm Mat}_P(\mathbb{C})$ as follows.
\begin{definition}\label{def:s}
For $x,y\in P$ we give the $(x,y)$-entry of $S$. Write $x\in P_i$ and $y\in P_j$ with $0\leq i, j\leq N$. Then
$$
S_{xy}= \left\{\begin{array} {cl}
1, & {\rm if}~~i=j~{\rm and}~x+y\in \tilde{P}_{i+1};\\
0, & {\rm  otherwise}.
\end{array}\right.
$$
\end{definition}

It is clear that $S^t=S$.
\begin{lemma}\label{lem:relationr2}
The matrix $S$ is related to $R, L, K$ as follows:
\begin{equation}\label{relation2:1}
S+LR-RL=\frac{K-q^{N+M}K^{-1}+(1-q^M)I}{q-1},
\end{equation}
\begin{equation}\label{relation2:2}
LS-qSL=(q^M-1)L,
\end{equation}
\begin{equation}\label{relation2:3}
SR-qRS=(q^M-1)R.
\end{equation}
\end{lemma}
\begin{proof}
We first obtain (\ref{relation2:1}).
To do this, for $x,y\in P$ compute the
$(x,y)$-entry of each term. We have

\medskip
\centerline{
\begin{tabular}[t]{c|c|l}
$(x,y)$-entry of $LR$ &$(x,y)$-entry of $RL$  & condition
   \\  \hline
 $(q-1)^{-1}(q^{N+M-i}-q^M)$ &$(q-1)^{-1}(q^i-1)$ &  $x=y\in P_i$
  \\
   $1$ & $1$ & $x, y\in P_i, \quad x+y\in P_{i+1}$
   \\
   $0$ & $1$ & $x, y\in P_i, \quad x+y\in \tilde{P}_{i+1}$
   \\
$0$ & $0$ & otherwise
 \\
 \end{tabular}}

 \medskip
\noindent
By the above arguments,

\medskip
\centerline{
\begin{tabular}[t]{c|l}
$(x,y)$-entry of $LR-RL$ & condition
   \\  \hline
 $(q-1)^{-1}(q^{N+M-i}-q^M-q^i+1)$ &  $x=y\in P_i$
  \\
 $-1$ & $x, y\in P_i, \quad  x+y\in \tilde{P}_{i+1}$
   \\
$0$ &  otherwise
 \\
 \end{tabular}}
 \medskip
\noindent
Equation (\ref{relation2:1}) follows from this along with (\ref{k}), (\ref{kinverse}) and Definition~\ref{def:s}.
Combining (\ref{relation:1}), (\ref{relation:2}), (\ref{relation2:1}) we obtain (\ref{relation2:2}).
 In (\ref{relation2:2}),  apply the transpose to each side and use $S^t=S, L^t=R$  to get  (\ref{relation2:3}).
\end{proof}

\begin{lemma}\label{lem:rlscommute}
The matrix $S$ commutes with each of $RL, LR, K, K^{-1}$.
\end{lemma}
\begin{proof}
Combining Lemma~\ref{lem:rlcommute} and (\ref{relation2:1}), we find that
 $S$ commutes with $K, K^{-1}$. Combining
(\ref{relation2:2}), (\ref{relation2:3}) we find that $S$ commutes with $RL, LR$.
 \end{proof}

\section{The irreducible $T$-modules and the center $Z(T)$}
In this section, we recall some facts from \cite{ter5, ter6},  about the irreducible
$T$-modules and the center $Z(T)$.

By a  $T$-module we mean a $T$-submodule
of $\mathbb{C}P$.
Let $W$ be an irreducible $T$-module.
By the {\em endpoint} of $W$ we mean min$\{i|0\leq i\leq N, F_iW\neq 0\}$.
By the {\em diameter} of $W$ we mean $|\{i|0\leq i\leq N, F_iW\neq 0\}|-1$.

\begin{lemma}{\rm\cite[Theorems~2.5, 3.3]{ter5}}\label{lem:bounds}
For $0\leq r,d\leq N$, there exists an  irreducible $T$-module with endpoint $r$ and diameter $d$, if and only if
$$ N-2r\leq d\leq N-r,\qquad\qquad d\leq N+M-2r.$$

\end{lemma}

\begin{lemma}{\rm\cite[Theorem~2.5]{ter5}}\label{rd}
Let $W$ be an irreducible $T$-module with endpoint $r$
 and diameter $d$. Then the isomorphism class of $W$
  is determined by the sequence $(r,d)$.
\end{lemma}
Let $Z(T)$ denote the center of $T$. We now give a basis for the vector space $Z(T)$.

Let $\Psi$ denote the set of isomorphism classes of
 irreducible $T$-modules. The elements of $\Psi$
  are called {\em types}.
 By Lemmas~\ref{lem:bounds}, \ref{rd}, we view
$$
\Psi=\{(r,d)\mid 0\leq r,d\leq N,\quad N-2r\leq d\leq N-r,\quad d\leq N+M-2r\}.
$$
For $\lambda=(r,d)\in \Psi$,
let $V_\lambda$ denote the subspace of $\mathbb{C}P$ spanned
by the irreducible $T$-modules of type $\lambda$.
Then $V_\lambda$ is a $T$-module,
and the sum $\mathbb{C}P=\sum_{\lambda\in \Psi}V_\lambda$ is  direct.
 For $\lambda\in \Psi$,
 define a linear map $e_\lambda: \mathbb{C}P\rightarrow \mathbb{C}P$  such that
 $(e_\lambda-I)V_\lambda=0$ and
$e_\lambda V_{\mu}=0$ for all $\mu\in \Psi$ with $\mu\neq\lambda$.
 According to the Wedderburn theory \cite{Curtis},
the elements $\{e_\lambda\}_{\lambda\in \Psi}$
form a basis for  $Z(T)$.

\begin{definition}{\rm\cite[line(16)]{ter6}}\label{def:notation}
\rm For $(r,d)\in \Psi$ define
\begin{eqnarray*}\label{xird}
x_{r+i}(r,d)=\frac{q^{N+M-r-d}(q^{i}-1)(q^{d+1-i}-1)}{(q-1)^2}\qquad\qquad\qquad (1\leq i\leq d).
\end{eqnarray*}
\end{definition}

\begin{note}\label{note:3}
{\rm Referring to Definition~\ref{def:notation}, $x_{r+i}(r,d)\neq 0$ for $1\leq i\leq d$.}
\end{note}

\begin{lemma}{\rm\cite[p.78]{ter6}\label{lem:irredbasis}}
Let $W$ denote an irreducible $T$-module with endpoint $r$ and diameter $d$.
Then there exists a basis  $\{w_i\}^d_{i=0}$ for $W$ such that
\begin{description}
  \item[\rm(i)]  $w_i\in F_{r+i}W\qquad\qquad (0\leq i\leq d)$,
  \item[\rm(ii)]  $Rw_i=w_{i+1}\qquad\qquad (0\leq i\leq d-1)$,  \qquad\qquad$Rw_d=0$,
\item[\rm(iii)] $Lw_i=x_{r+i}(r,d)w_{i-1}\qquad\qquad (1\leq i\leq d)$, \qquad\qquad$Lw_0=0$.
\end{description}
\end{lemma}


\section{The central elements $\Phi, \Omega$}
In this section, we display two  elements $\Phi, \Omega$ in $Z(T)$.
Define
$$\Phi=\sum\limits_{\lambda=(r,d)\in \Psi}q^{r} e_\lambda,
\qquad\qquad \Omega=\sum\limits_{\lambda=(r,d)\in \Psi}q^{d/2} e_\lambda.$$
By construction $\Phi, \Omega$ are in $Z(T)$.
Note that $\Phi, \Omega$ are invertible, with
$$\Phi^{-1}=\sum\limits_{\lambda=(r,d)\in \Psi}q^{-r} e_\lambda,
\qquad\qquad \Omega^{-1}=\sum\limits_{\lambda=(r,d)\in \Psi}q^{-d/2} e_\lambda.$$

\begin{lemma}\label{rem:phiomeaction}
Let $W$ denote an irreducible $T$-module with endpoint $r$ and diameter $d$. Then on $W$,
$$
\Phi=q^{r}I,\qquad\qquad\qquad \Omega=q^{d/2}I.$$
\end{lemma}
\begin{proof}
By construction.
\end{proof}

\begin{proposition}\label{lem:centralgenerate}
The center $Z(T)$ is generated by $\Phi, \Omega$.
\end{proposition}
\begin{proof}
Let $Z$ denote the subalgebra of $Z(T)$ generated by $\Phi, \Omega$. We show
$Z=Z(T)$. To do this,
it suffices to show that $e_\lambda\in Z$ for all $\lambda\in \Psi$.
Suppose we are given two distinct elements in $\Psi$, denoted by
$(r,d)$ and  $(r',d')$. Observe that $q^r\neq q^{r'}$ or  $q^{d/2}\neq q^{d'/2}$.
By this and Lemma~\ref{rem:phiomeaction}, we see that
 $e_\lambda\in Z$ for all $\lambda\in \Psi$.
\end{proof}


\section{The central elements $C_1, C_2$}
In this section, we display two  elements $C_1, C_2$ in $Z(T)$ and discuss their  combinatorial meaning.

\begin{definition}\label{def:central}
Define
\begin{eqnarray*}
& &C_1=q^{-1}(q-1)^{-1}(q+1)K+RL-q^{-1}LR,\\
& &C_2=K^2+(q-1)RLK-(q-1)LRK.
\end{eqnarray*}
\end{definition}
\begin{lemma}\label{lem:transpose}
We have $C^t_1=C_1$ and $C^t_2=C_2$.
\end{lemma}
\begin{proof}
By $R^t=L$ along with Lemma~\ref{lem:rlcommute} and Definition~\ref{def:central}.
\end{proof}
\begin{lemma}\label{lem:c1c2}
The matrices $C_1$ and $C_2$ are  in  $Z(T)$.
\end{lemma}
\begin{proof}
We first show that $C_1\in Z(T)$. To do this, we show that $C_1$ commutes with each of
$R, L, K$. To see that $C_1$ commutes with $R$, use $RK=qKR$ and (\ref{relation:3}).
To see that $C_1$ commutes with $L$, use $R^t=L$ and Lemma~\ref{lem:transpose}.
The matrix $C_1$ commutes with $K$ by Lemma~\ref{lem:rlcommute}. We have shown that $C_1\in Z(T)$.
 In a similar way, one can show that
$C_2\in Z(T)$.
\end{proof}

\begin{lemma}\label{lem:c1c2action}
Let $W$ denote an irreducible $T$-module  with endpoint $r$ and diameter $d$. Then on $W$,
$$
C_1=(q-1)^{-1}q^{N+M-r}(1+q^{-d-1})I,\quad\quad C_2=q^{2N+2M-2r-d}I.$$
\end{lemma}
\begin{proof}
We consider the actions of $C_1$ and $C_2$ on the basis $\{w_i\}^d_{i=0}$ given
in Lemma~\ref{lem:irredbasis}.
By  the definition of $K$ along with  Definition~\ref{def:notation},
Lemma~\ref{lem:irredbasis}, Definition~\ref{def:central}, we find that
for $0\leq i\leq d$,
\begin{eqnarray}
C_1w_i&=&(q-1)^{-1}q^{N+M-r}(1+q^{-d-1})w_i,\label{c1action}
\\
C_2w_i&=&q^{2N+2M-2r-d}w_i.\label{c2action}
\end{eqnarray}
The result follows.
\end{proof}

\begin{proposition}\label{thm:main1}
The center $Z(T)$ is generated by $C_1$, $C_2$.
\end{proposition}
\begin{proof}
The proof is similar to that of Lemma~\ref{lem:centralgenerate}.
Let $Z$ denote the subalgebra of $Z(T)$ generated by $C_1$, $C_2$. We show that
$Z=Z(T)$.
To do this, it suffices to show that  $e_\lambda\in Z$ for all
$\lambda\in \Psi$.
Suppose we are given two distinct elements in $\Psi$, denoted by
$(r,d)$ and  $(r',d')$. We claim that
\begin{eqnarray*}
q^{-r}+q^{-r-d-1}&\neq &q^{-r'}+q^{-r'-d'-1}
\end{eqnarray*}
or
\begin{eqnarray*}
q^{-2r-d}&\neq &q^{-2r'-d'}.
\end{eqnarray*}
Suppose the claim is false. We have
\begin{eqnarray}
q^{-r}+q^{-r-d-1}&= &q^{-r'}+q^{-r'-d'-1},\label{rrdd1}\\
q^{-2r-d}&= &q^{-2r'-d'}.\label{rrdd2}
\end{eqnarray}
Suppose for the moment that $r=r'$. Then by (\ref{rrdd2}),
$q^{d-d'}=1$. But $q$ is not a root of unity, so $d=d'$ for
a contradiction. Therefore $r\neq r'$. Without loss of generality
we may assume $r<r'$. Eliminating $q^{-d}$ in (\ref{rrdd1}) using (\ref{rrdd2})
we find
$$
(1-q^{r-r'})(q^{-r}-q^{-r'-d'-1})=0.
$$
Note that $1-q^{r-r'}$ is nonzero, so  $q^{-r}=q^{-r'-d'-1}$.
Therefore $r=r'+d'+1$, so
\begin{equation}\label{hh1}
d'=r-r'-1.
\end{equation}
Now $d'<0$ for a contradiction. The claim is proved. By the claim
 and Lemma~\ref{lem:c1c2action}, $e_\lambda\in Z$ for all $\lambda\in \Psi$. The result follows.
\end{proof}

\begin{lemma}\label{lem:rlc1c2}
The following {\rm(i)}, {\rm(ii)} hold.
\begin{description}
  \item[\rm(i)] $RL=-q(q-1)^{-2}K+q(q-1)^{-1}C_1-(q-1)^{-2}C_2K^{-1}$,
  \item[\rm(ii)]  $LR=-(q-1)^{-2}K+q(q-1)^{-1}C_1-q(q-1)^{-2}C_2K^{-1}$.
\end{description}
\end{lemma}
\begin{proof}
Solve the two equations in Definition~\ref{def:central} for $RL, LR$.
\end{proof}

We now discuss how  $C_1, C_2$ are related to $\Phi, \Omega$.
\begin{lemma}\label{lem:relation}
We have
\begin{eqnarray}
C_1&=&q^{N+M-1}\Phi^{-1}\Omega^{-1}\frac{q^{1/2}\Omega+q^{-1/2}\Omega^{-1}}{q^{1/2}-q^{-1/2}},\\
C_2&=&q^{2N+2M}\Phi^{-2}\Omega^{-2}.
\end{eqnarray}
\end{lemma}
\begin{proof}
Combine Lemma~\ref{rem:phiomeaction} and Lemma~\ref{lem:c1c2action}.
\end{proof}

We now interpret our results so far in terms of augmented down-up algebras \cite{terwora}.
Fix a nonzero  $\tau\in \mathbb{C}$ which is not a root of unity .

\begin{definition}\rm\cite[Definition 2.3]{terwora}\label{terwora}
Let $s,t$ denote distinct integers. Let $\phi$ denote a Laurent polynomial over $\mathbb{C}$  in a variable $\vartheta$.
The {\em augmented down-up algebra} $\mathbb{A}_\tau(s,t,\phi)$ is the
 $\mathbb{C}$-algebra
with generators ${\cal K}^{\pm 1}, E, F, {\cal C}_s, {\cal C}_t$
and relations
$$
\begin{array}{l}
{\cal K}{\cal K}^{-1}={\cal K}^{-1}{\cal K}=1,\\
{\cal C}_s, {\cal C}_t~~{\rm are ~central},\\
{\cal K}E=\tau^2E{\cal K},\qquad ~~{\cal K}F=\tau^{-2}F{\cal K},\\
FE={\cal C}_s\tau^s{\cal K}^s+{\cal C}_t\tau^t{\cal K}^t+\phi(\tau{\cal K}),\\
EF={\cal C}_s\tau^{-s}{\cal K}^s+{\cal C}_t\tau^{-t}{\cal K}^t+\phi(\tau^{-1}{\cal K}).
\end{array}
$$
\end{definition}

\begin{lemma}\label{lem:augmented}
Let $\tau=q^{1/2}$. The vector space $\mathbb{C}P$ has an $\mathbb{A}_{\tau}(s,t,\phi)$-module structure on which
${\cal K}, E, F, {\cal C}_s, {\cal C}_t$ act as follows:

\medskip
\centerline{
\begin{tabular}[t]{c|c c c c c}
{\rm generator} &${\cal K}$ &$E$ & $F$ & ${\cal C}_s$ & ${\cal C}_t$
   \\  \hline
 {\rm action} & $K$ & $L$ & $R$ & $-\frac{\tau}{(\tau^2-1)^2}C_2$ &$\frac{\tau^2}{\tau^2-1}C_1$
  \\
 \end{tabular}}

 \noindent
Here $s=-1$, $t=0$, $\phi=-\tau(\tau^2-1)^{-2}\vartheta$.
\end{lemma}
\begin{proof}
Combine (\ref{relation:1}),  Lemmas~\ref{lem:c1c2}, \ref{lem:rlc1c2} and Definition~\ref{terwora}.
\end{proof}
\begin{remark}\rm
The  $\mathbb{A}_\tau(s,t,\phi)$-module structure in Lemma~\ref{lem:augmented} is discussed further  in \cite{terwora}.
\end{remark}


\section{Some $U_\tau({\mathfrak{sl}}_2)$-module structures on $\mathbb{C}P$}
Throughout this section, fix a nonzero  $\tau\in \mathbb{C}$ that is not a root of unity.
In this section    we recall the quantum enveloping algebra
$U_\tau({\mathfrak{sl}}_2)$.  We then  display  some $U_\tau({\mathfrak{sl}}_2)$-module structures on $\mathbb{C}P$.

\begin{definition}\rm\cite{kassel}  \label{def:uqsl2}
Let $U_\tau({\mathfrak{sl}}_2)$ denote the
$\mathbb{C}$-algebra with generators $e,f, k^{\pm 1}$ and the following relations:
\begin{equation}\label{uqsr1}
kk^{-1}=k^{-1}k=1,
\end{equation}
\begin{equation}\label{uqsr2}
ke=\tau^2ek,\qquad \qquad  kf=\tau^{-2}fk,
\end{equation}
\begin{equation} \label{uqsr3}
ef-fe=\frac{k-k^{-1}}{\tau-\tau^{-1}}.
\end{equation}
We call $e,f, k^{\pm 1}$ the {\em Chevalley generators} for $U_\tau({\mathfrak{sl}}_2)$.
\end{definition}

\begin{theorem}\label{thm:main2}
Assume that $\tau=q^{1/2}$ and let $\Theta$ denote an invertible element in $Z(T)$. Then $\mathbb{C}P$ becomes a
$U_\tau({\mathfrak{sl}}_2)$-module on which $e, f, k$ act as follows:

\medskip
\centerline{
\begin{tabular}[t]{c|c c c}
{\rm generator}  & $e$ & $f$ & $k$
   \\  \hline
 {\rm action}  & $\Theta L$ & $q^{-N-M+\frac{1}{2}}\Theta^{-1}\Phi\Omega R$ &  $q^{-N-M}\Phi\Omega K$
  \\
 \end{tabular}}
\end{theorem}
\begin{proof}
Since $K$ is invertible,  (\ref{uqsr1}) holds. By (\ref{relation:1})  we obtain  (\ref{uqsr2}).
To obtain (\ref{uqsr3}), combine Lemmas~\ref{lem:rlc1c2}, \ref{lem:relation}.
\end{proof}


\section{Some results about $R, L$}
For the rest of this paper, assume that $N\geq 6$.
In this section, we obtain some results about $R, L$ that we will use later in the paper.

\begin{lemma}\label{lem:1602271} The following hold in the algebra $T$:
\begin{eqnarray}
& & LR^3F_0=(q-1)^{-2}q^M(q^3-1)(q^{N-2}-1)R^2F_0,\label{201602271}\\
& & L^3RF_2=(q-1)^{-2}q^M(q^3-1)(q^{N-2}-1)L^2F_2.\label{201602272}
\end{eqnarray}
\end{lemma}
\begin{proof}
We first obtain (\ref{201602271}). Let $y$ be the unique element in $P_0$. For  $x\in P_{2}$, we calculate  the $(x,y)$-entry of each matrix:

\bigskip
\centerline{
\begin{tabular}[t]{c| c}
$(x,y)$-entry of $R^2F_0$&$(x,y)$-entry of $LR^3F_0$
   \\  \hline
$q+1$ & $(q-1)^{-2}q^M(q+1)(q^3-1)(q^{N-2}-1)$
   \\
 \end{tabular}}

 \bigskip
 \noindent
Line (\ref{201602271}) follows.  Apply the transpose map to each side of (\ref{201602271}) to get (\ref{201602272}).
\end{proof}

\begin{lemma}\label{lem:rli31}
In the poset $\mathcal{A}_q(N,M)$,
\begin{description}
\item[\rm(i)] there exists  $x\in P_{N}$ and $y\in P_{N-2}$ such that $x+y\in \tilde{P}_{N+1}$,
\item[\rm(ii)] there exists  $x\in P_{N}$ and $y\in P_{N-2}$ such that $y<x.$
\end{description}
 \end{lemma}
\begin{proof}
By the arguments about $\mathcal{A}_q(N,M)$ given in Section 2.
\end{proof}

\begin{lemma}\label{lem:rli}
For $1\leq i\leq N-3$,
\begin{description}
\item[\rm(i)] there exists  $x\in P_{i+2}$ and  $y\in P_i$ such that $x+y\in P_{i+3}$,
\item[\rm(ii)] there exists  $x\in P_{i+2}$ and $y\in P_i$ such that $x+y\in \tilde{P}_{i+3}$,
\item[\rm(iii)] there exists  $x\in P_{i+2}$ and $y\in P_i$ such that $y<x.$
\end{description}
 \end{lemma}
\begin{proof}
By the arguments about $\mathcal{A}_q(N,M)$ given in Section 2.
\end{proof}

\begin{lemma}\label{lem:rli311}
The following hold in the algebra $T$:
\begin{description}
  \item[\rm(i)] $R^2F_{N-2}, R^3LF_{N-2}$ are  linearly independent,
  \item[\rm(ii)] $L^2F_{N}, RL^3F_{N}$ are  linearly independent.
\end{description}
 \end{lemma}
\begin{proof}
(i) For    $x\in P_{N}$ and $y\in P_{N-2}$, we calculate  the $(x,y)$-entry of each matrix:

\medskip
\centerline{
\begin{tabular}[t]{c c|l}
$(x,y)$-entry of $R^2F_{N-2}$ &$(x,y)$-entry of $R^3LF_{N-2}$ & condition
   \\  \hline
 $q+1$ & $(q-1)^{-2}(q+1)(q^3-1)(q^{N-2}-1)$ &  $y< x$
 \\
 $0$ & $(q-1)^{-1}(q+1)(q^3-1)$ & $x+y\in \tilde{P}_{N+1}$
   \\
 \end{tabular}}
\medskip
\noindent
In the above table the entries form a $2\times 2$ matrix. This matrix is upper triangular with nonzero diagonal entries, so it is invertible. The result follows in view of Lemma~\ref{lem:rli31}.

\noindent
(ii) Apply the transpose map to the matrices in (i) and use $R^t=L$ along with (\ref{12153}), (\ref{12154}).
\end{proof}

\begin{lemma}\label{lem:rli0}
For $1\leq i\leq N-3$, the following hold in the algebra $T$:
\begin{description}
  \item[\rm(i)] $R^2F_i, R^3LF_i, LR^3F_i$ are  linearly independent;
  \item[\rm(ii)] $L^2F_{i+2}, L^3RF_{i+2}, RL^3F_{i+2}$ are  linearly independent.
\end{description}
 \end{lemma}
\begin{proof}
(i) For    $x\in P_{i+2}$ and $y\in P_i$, we calculate  the $(x,y)$-entry of each matrix:

\bigskip
\centerline{
\begin{tabular}[t]{c c c|l}
$(x,y)$-entry of $R^2F_i$&$(x,y)$-entry of $LR^3LF_i$ &$(x,y)$-entry of $R^3LF_i$ & condition
   \\  \hline
$q+1$& $(q-1)^{-1}q^M(q^{N-i-2}-1)\gamma$ & $(q-1)^{-1}(q^{i}-1)\gamma$ & $y< x$
  \\
$0$ & $\gamma$ & $\gamma$ &  $x+y\in P_{i+3}$
 \\
$0$ & $0$ & $\gamma$ & $x+y\in \tilde{P}_{i+3}$
   \\
 \end{tabular}}
\noindent
where
\begin{eqnarray*}
\gamma=(q-1)^{-1}(q+1)(q^3-1).
\end{eqnarray*}
\noindent
In the above table the entries form a $3\times 3$ matrix. This matrix is upper triangular with nonzero diagonal entries, so it is invertible. The result follows in view of Lemma~\ref{lem:rli}.

\noindent
(ii) Apply the transpose map to the matrices in (i) and use $R^t=L$ along with (\ref{12153}), (\ref{12154}).
\end{proof}


\section{Leonard pairs}

In this section, we recall the definition of a Leonard pair
and discuss some basic facts about these objects.

Through this section, fix an integer  $d\geq 0$.  Let $V$ denote
 a  vector space over $\mathbb{C}$ with dimension $d+1$.
Denote by ${\rm End}(V)$  the $\mathbb{C}$-algebra of all
$\mathbb{C}$-linear maps  $V\rightarrow V$.
Let $\{v_i\}_{i=0}^d$ denote a basis of $V$. For
  $A\in {\rm End}(V)$  and $X\in {\rm Mat}_{d+1}(\mathbb{C})$,
  we say that $X$ {\em represents $A$
with respect to} $\{v_i\}_{i=0}^d$ whenever
$Av_j=\sum_{i=0}^d X_{ij}v_i$ for $0\leq j\leq d$.
For $X\in {\rm Mat}_{d+1}(\mathbb{C})$, $X$ is called
{\em upper bidigonal}
whenever every nonzero entry appears on the diagonal or the superdiagonal.
The matrix $X$ is called
{\em lower bidigonal}
whenever every nonzero entry appears on the diagonal or  the subdiagonal. The matrix $X$ is called {\em tridiagonal}
  whenever every nonzero entry appears  on the diagonal, the superdiagonal, or the
subdiagonal. Assume that $X$ is tridiagonal. Then $X$ is
 called {\em irreducible} whenever
the  entries on the superdiagonal and subdiagonal  are all nonzero.

\begin{definition}\rm\cite[Definition 1.1]{ter7+}\label{def:leonard}
By a {\em Leonard pair} on $V$, we mean an ordered pair
$A, A^*$ of elements in ${\rm End}(V)$ that satisfy
the following conditions:
\begin{enumerate}
  \item[(i)] there exists a basis for $V$ with respect to
   which the matrix representing $A$ is irreducible tridiagonal
   and the matrix representing $A^*$ is diagonal;
  \item[(ii)] there exists a basis for $V$
  with respect to which the matrix representing $A^*$
  is irreducible tridiagonal and the matrix representing $A$ is diagonal.
\end{enumerate}
We call $V$ the {\em underlying vector space}, and call $d$  the {\em diameter}.
\end{definition}

\begin{definition}\rm
Referring to  Definition~\ref{def:leonard}, consider a basis for $V$ from part (ii).
With respect to this basis, the matrix representing $A$ is diagonal, denoted by
$\rm{diag}(\theta_0, \theta_1, \ldots, \theta_d)$. We call $\{\theta_i\}_{i=0}^d$
an {\em eigenvalue sequence} for $A, A^*$. By  a {\em dual eigenvalue sequence}
for $A, A^*$, we mean an eigenvalue sequence for the Leonard pair $A^*, A$.
\end{definition}

\begin{note}\rm
Let $A, A^*$ denote a Leonard pair on $V$.
Let $\{\theta_i\}_{i=0}^d$ denote an eigenvalue sequence for $A, A^*$. Then
the sequence $\{\theta_{d-i}\}_{i=0}^d$ is an eigenvalue sequence for
$A, A^*$ and $A, A^*$ has no other eigenvalue sequence. A similar comment
applies to the dual eigenvalue sequence.
\end{note}

\begin{definition}{\rm \label{def:nonzero}{\cite[Definition 22.1]{ter7++}}}
\rm By a {\em parameter array over $\mathbb{C}$ of diameter $d$} we mean a sequence $(\{\theta_i\}_{i=0}^d, \{\theta^*_i\}_{i=0}^d,
 \{\varphi_j\}^d_{j=1}, \{\phi_j\}^d_{j=1})$ of elements in $\mathbb{C}$ that satisfy the following conditions:
\begin{enumerate}
  \item[\rm(i)] $\theta_i\neq \theta_j,\qquad \theta^*_i\neq \theta^*_j$ \qquad {\rm if}~ $i\neq j$\qquad \qquad ($0\leq i,j\leq d$),
  \item[\rm(ii)] $\varphi_i\neq 0,\qquad \phi_i\neq 0$ \qquad\qquad\qquad ($1\leq i\leq d$),
  \item[\rm(iii)] $\varphi_i=\phi_1\sum\limits_{h=0}^{i-1}\frac{\theta_h-\theta_{d-h}}{\theta_0-\theta_d}+(\theta^*_i-\theta^*_0)(\theta_{i-1}-\theta_d)$\qquad\qquad\qquad ($1\leq i\leq d$),
  \item[\rm(iv)] $\phi_i=\varphi_1\sum\limits_{h=0}^{i-1}\frac{\theta_h-\theta_{d-h}}{\theta_0-\theta_d}+(\theta^*_i-\theta^*_0)(\theta_{d-i+1}-\theta_0)$\qquad\qquad\qquad ($1\leq i\leq d$),
  \item[\rm(v)] the expressions
  $$
  \frac{\theta_{i-2}-\theta_{i+1}}{\theta_{i-1}-\theta_i},\qquad\qquad\qquad\frac{\theta^*_{i-2}-\theta^*_{i+1}}{\theta^*_{i-1}-\theta^*_i}
  $$
  are equal and independent of $i$ for $2\leq i\leq d-1$.
\end{enumerate}
\end{definition}

\begin{lemma}{\rm\cite[Theorem 17.1]{ter9}\label{thm:leoard equivalent}}
Let $A, A^*$ denote matrices in
${\rm Mat}_{d+1}(\mathbb{C})$. Assume that $A$ is  lower bidiagonal and $A^*$ is upper bidiagonal. Then the following {\rm(i)}, {\rm(ii)} are equivalent:
\begin{enumerate}
\item[\rm(i)] the pair $A, A^*$ is a Leonard pair,

\item[\rm(ii)] there exists a parameter array $(\{\theta_i\}_{i=0}^d, \{\theta^*_i\}_{i=0}^d,
 \{\varphi_j\}^d_{j=1}, \{\phi_j\}^d_{j=1})$ over $\mathbb{C}$ such that
$$
A_{ii}=\theta_i, \qquad\qquad A^*_{ii}=\theta^*_i\qquad\qquad\qquad(0\leq i\leq d),
$$
$$
A_{i,i-1}A^*_{i-1,i}=\varphi_i\qquad\qquad\qquad(1\leq i\leq d).
$$
\end{enumerate}
\end{lemma}

\begin{lemma}{\rm\cite[Theorem 10.1]{itotanater}\label{lem:tdrelation}}
Let $A, A^*$ denote a Leonard pair over  $\mathbb{C}$.
Then there exists a sequence of scalars
$\beta, \gamma, \gamma^*, \varrho, \varrho^*$ taken
from $\mathbb{C}$ such that both

\begin{equation}\label{equa11}
[A, A^2A^*-\beta AA^*A+A^*A^2-\gamma(AA^*+A^*A)-\varrho A^*]=0,
\end{equation}
\begin{equation}\label{equa22}
[A^*, A^{*2}A-\beta A^*AA^*+AA^{*2}-\gamma^*(A^*A+AA^*)-\varrho^* A]=0,
\end{equation}
where $[r,s]$ means $rs-sr$. The sequence is uniquely determined
by the Leonard pair $A, A^*$ provided that the diameter is at least $3$.
\end{lemma}
We call (\ref{equa11}), (\ref{equa22}) the {\em tridiagonal relations}.

\begin{lemma}{\rm\cite[Corollary~4.4, Theorem~4.5]{tervid}\label{lem:leonardiff}}
Let $A, A^*$ denote a Leonard pair over  $\mathbb{C}$.
Let $\{\theta_i\}_{i=0}^d$ (resp. $\{\theta^*_i\}_{i=0}^d$) denote an eigenvalue sequence
(resp. dual eigenvalue sequence) of $A, A^*$.
Let $\beta, \gamma, \gamma^*, \varrho, \varrho^*$
denote scalars in $\mathbb{C}$. Then these scalars satisfy
{\rm(\ref{equa11}), (\ref{equa22})}    if and only if the following
{\rm(i)--(v)} hold:
\begin{enumerate}
\item[\rm(i)]
$
\beta+1=\frac{\theta_{i-2}-\theta_{i+1}}{\theta_{i-1}-\theta_i}
=\frac{\theta^*_{i-2}-\theta^*_{i+1}}{\theta^*_{i-1}-\theta^*_i}\qquad\qquad\qquad (2\leq i\leq d-1),
$

\item[\rm(ii)] $\gamma=\theta_{i-1}-\beta \theta_i+\theta_{i+1}\qquad\qquad\qquad (1\leq i\leq d-1),$

\item[\rm(iii)] $\gamma^*=\theta^*_{i-1}-\beta \theta^*_i+\theta^*_{i+1}\qquad\qquad\qquad (1\leq i\leq d-1),$

\item[\rm(iv)] $\varrho=\theta^2_{i-1}-\beta \theta_{i-1}\theta_i+\theta^2_{i}-
\gamma (\theta_{i-1}+\theta_i)\qquad\qquad\qquad (1\leq i\leq d),$

\item[\rm(v)] $\varrho^*=\theta^{*2}_{i-1}-\beta \theta^*_{i-1}\theta^*_i+\theta^{*2}_{i}
-\gamma^* (\theta^*_{i-1}+\theta^*_i)\qquad\qquad\qquad (1\leq i\leq d).$

\end{enumerate}
\end{lemma}

Let $n$ denote a positive integer. For the rest of this section, let $\{\theta_i\}^n_{i=0}$ denote a sequence of scalars in $\mathbb{C}$.
For $\beta\in \mathbb{C}$,  $\{\theta_i\}_{i=0}^n$ is called
 {\em $\beta$-recurrent} whenever
$$
\theta_{i-2}-(\beta+1)\theta_{i-1}+(\beta+1)\theta_{i}-\theta_{i+1}=0 \qquad\qquad (2\leq i\leq n-1).
$$

\begin{lemma}{\rm\cite[Lemma~9.2]{ter7+}\label{lem:betarecurrence}}
 Given a sequence $\{\theta_i\}_{i=0}^{n}$ of scalars in $\mathbb{C}$ and
given  $\beta\in \mathbb{C}$.
\begin{description}
\item[\rm(i)] Assume  $\beta=2$.  Then $\{\theta_i\}_{i=0}^{n}$  is $\beta$-recurrent if and only if there exist scalars $a,b,c\in \mathbb{C}$ such that
$$
\theta_i=a+bi+ci^2\qquad\qquad (0\leq i\leq n).
$$
  \item[\rm(ii)] Assume  $\beta=-2$.  Then $\{\theta_i\}_{i=0}^{n}$  is $\beta$-recurrent if and only if there exist scalars $a,b,c\in \mathbb{C}$ such that
$$
\theta_i=a+b(-1)^i+ci(-1)^i\qquad\qquad (0\leq i\leq n).
$$
  \item[\rm(iii)]  Assume  $\beta\neq 2$, $\beta\neq-2$.  Then $\{\theta_i\}_{i=0}^{n}$  is $\beta$-recurrent
if and only if there exist scalars $a,b,c\in \mathbb{C}$ such that
$$
\theta_i=a+bQ^i+cQ^{-i}\qquad\qquad (0\leq i\leq n),
$$
where $\beta=Q+Q^{-1}$.
\end{description}
\end{lemma}


\section{The tridiagonal relations and  $\mathcal{A}_q(N, M)$}
We continue to discuss the poset $\mathcal{A}_q(N, M)$ from Section 2.
For the rest of the paper, we fix matrices $A, A^*$
of the following form:

\begin{equation}\label{leonard1}
A=\sum\limits_{i=0}^{N-1} \alpha_iRF_i+\sum\limits_{i=0}^N \theta_iF_i,
\end{equation}

\begin{equation}\label{leonard2}
A^*=\sum\limits_{i=1}^N \alpha^*_iLF_i+\sum\limits_{i=0}^N \theta^*_iF_i.
\end{equation}
The $\alpha_i$,  $\alpha^*_i$, $\theta_i$, $\theta^*_i$ are  scalars in $\mathbb{C}$ such that
\begin{eqnarray}
& &\alpha_i\neq0\qquad\qquad(0\leq i\leq N-1),\label{alph0}\\
& &\alpha^*_i\neq 0 \qquad\qquad(1\leq i\leq N),\label{alph}\\
& &\theta_i\neq\theta_j \qquad\qquad \theta^*_i\neq\theta^*_j \qquad\qquad {\rm if} ~~ i\neq j \qquad\qquad(0\leq i, j\leq N).\label{thetasatisfy}
\end{eqnarray}
The matrices $R, L$ are from (\ref{lr}) and $\{F_i\}_{i=0}^N$ are from
(\ref{fi}).

For  convenience, let
$$\alpha_{-1}=0,\qquad \alpha_{N}=0, \qquad  \alpha^*_{0}=0,\qquad \alpha^*_{N+1}=0.$$
Let $\theta_{-1}, \theta_{N+1}, \theta^*_{-1}, \theta^*_{N+1}$ denote indeterminates.
Define
\begin{equation}\label{xala}
\xi_i=\alpha_i\alpha^*_{i+1}\qquad\qquad \qquad(-1\leq i\leq N).
\end{equation}
Observe that $\xi_{-1}=0$, $\xi_N=0$ and $\xi_i\neq 0$ for $0\leq i\leq N-1.$

\noindent
Fix some scalars
$\beta, \gamma, \gamma^*, \varrho, \varrho^*$ in $\mathbb{C}$ and define
\begin{eqnarray}
& &B=[A, A^2A^*-\beta AA^*A+A^*A^2-\gamma(AA^*+A^*A)-\varrho A^*],\label{equa1}\\
& &B^*=[A^*, A^{*2}A-\beta A^*AA^*+AA^{*2}-\gamma^*(A^*A+AA^*)-\varrho^* A].\label{equa2}
\end{eqnarray}

We now find necessary and sufficient conditions for $B$ and $B^*$ to be zero.
Expand $B$ and $B^*$ as follows:
\begin{eqnarray}
B&=&A^3A^*-(\beta+1) A^2A^*A+(\beta+1)AA^*A^2-A^*A^3\nonumber
\\
& &-~\gamma(A^2A^*-A^*A^2)-\varrho (AA^*-A^*A),\label{equa1sim}
\\
B^*&=&A^{*3}A-(\beta+1) A^{*2}AA^*+(\beta+1)A^*AA^{*2}-AA^{*3}\nonumber
\\
& &-~\gamma^*(A^{*2}A-AA^{*2})-\varrho^* (A^*A-AA^*).\label{equa2sim}
\end{eqnarray}

\begin{lemma}\label{lem:1602241}
For $0\leq i,j\leq N$, we have
\begin{description}
  \item[{\rm(i)}] $F_iBF_j=0$ if $i-j<-1$ or $i-j>3$,
  \item[{\rm(ii)}] $F_iB^*F_j=0$ if $j-i<-1$ or $j-i>3$.
\end{description}
\end{lemma}
\begin{proof}
Evaluating (\ref{equa1sim}), (\ref{equa2sim}) using (\ref{leonard1}), (\ref{leonard2}).
\end{proof}

\begin{lemma}\label{lem:btri} In the algebra $T$,
\begin{enumerate}
\item[\rm (i)]
for $0\leq i\leq N-3$, $F_{i+3}BF_i$ is equal to  $\alpha_{i}\alpha_{i+1}\alpha_{i+2}\bigl(\theta^*_i-\theta^*_{i+3}-(\beta+1)(\theta^*_{i+1}-\theta^*_{i+2})\bigr)$
times $R^3F_i$,

\item[\rm (ii)]
for $1\leq i\leq N$, $F_{i-1}BF_i$ is  equal to $\alpha^*_{i}(\theta_{i-1}-\theta_i)\bigl(\theta^2_{i-1}
-\beta\theta_{i-1}\theta_i+\theta^2_i-\gamma(\theta_{i-1}+\theta_i)
-\varrho\bigr)$ times $LF_i$,

\item[\rm (iii)]
for $0\leq i\leq N-1$, $F_{i+1}BF_i$ is  a weighted sum involving the following
terms and coefficients:

\medskip

\centerline{
\begin{tabular}[t]{c|l}\label{fl-1f}
{\rm term} & {\rm coefficient}
   \\  \hline
$RF_i$ & $\alpha_i(\theta^*_i-\theta^*_{i+1})\bigl(\theta^2_{i}
-\beta\theta_{i}\theta_{i+1}+\theta^2_{i+1}-\gamma(\theta_{i}+\theta_{i+1})
-\varrho\bigr)$
  \\
$R^2LF_i$ & $\alpha_i\alpha_{i-1}\alpha^*_i(\theta_{i-1}-\beta \theta_i
+\theta_{i+1}-\gamma)$
   \\
 $LR^2F_i$ &  $\alpha_i\alpha_{i+1}\alpha^*_{i+2}(-\theta_{i}+\beta \theta_{i+1}-\theta_{i+2}+\gamma)$
 \\
\end{tabular}}

\item[\rm (iv)] for $0\leq i\leq N$,
$F_{i}BF_i$ is a weighted sum  involving the following terms and coefficients:

\medskip

\centerline{
\begin{tabular}[t]{c|l}\label{fl0f}
{\rm term} & {\rm coefficient}
   \\  \hline
$RLF_i$ & $\alpha^*_i\alpha_{i-1}\bigl(\theta^2_{i-1}
-\beta\theta_{i-1}\theta_{i}+\theta^2_{i}-\gamma(\theta_{i-1}+\theta_{i})
-\varrho\bigr)$
  \\
$LRF_i$ & $-\alpha^*_{i+1}\alpha_i\bigl(\theta^2_i
-\beta\theta_i\theta_{i+1}+\theta^2_{i+1}-\gamma(\theta_i+\theta_{i+1})
-\varrho\bigr)$
 \\
\end{tabular}}

\item[\rm (v)] for $0\leq i\leq N-2$,
$F_{i+2}BF_i$ is a weighted sum  involving the following terms and coefficients:

\medskip
\centerline{
\begin{tabular}[t]{c|l}
{\rm term} & {\rm coefficient}
   \\  \hline
$R^2F_i$ & $\alpha_i\alpha_{i+1}\Big((\theta^*_i-\theta^*_{i+2})(\theta_{i}
+\theta_{i+1}+\theta_{i+2}-\gamma)
-(\beta+1)\bigl((\theta^*_{i+1}-\theta^*_{i+2})\theta_{i+2}
+(\theta^*_i-\theta^*_{i+1})\theta_i\bigr)\Big)$
  \\
$R^3LF_i$ & $\alpha_i\alpha_{i+1}\alpha_{i-1}\alpha^*_i$
   \\
 $R^2LRF_i$ & $-(\beta+1)\alpha^2_i\alpha_{i+1}\alpha^*_{i+1}$
\\
 $RLR^2F_i$ & $(\beta+1)\alpha_i\alpha^2_{i+1}\alpha^*_{i+2}$
 \\
$LR^3F_i$ & $-\alpha_i\alpha_{i+1}\alpha_{i+2}\alpha^*_{i+3}$
\end{tabular}}
\end{enumerate}
\end{lemma}

\begin{proof}
Combine  (\ref{12151}), (\ref{12153}), (\ref{12154}), (\ref{leonard1}), (\ref{leonard2}), (\ref{equa1sim}).
\end{proof}

\begin{lemma}\label{lem:bstartri}
In the algebra $T$,
\begin{enumerate}
\item[\rm (i)]
for $3\leq i\leq N$, $F_{i-3}B^*F_i$ is equal to $\alpha^*_{i}\alpha^*_{i-1}\alpha^*_{i-2}
\bigl(\theta_i-\theta_{i-3}-(\beta+1)(\theta_{i-1}-\theta_{i-2})\bigr)$ times $L^3F_i$,

\item[\rm (ii)]
for $0\leq i\leq N-1$, $F_{i+1}B^*F_i$ is equal to $\alpha_{i}(\theta^*_{i+1}-\theta^*_i)(\theta^{*2}_{i+1}
-\beta\theta^*_{i+1}\theta^*_i+\theta^{*2}_i-\gamma^*(\theta^*_{i+1}+\theta^*_i)
-\varrho^*)$ times $RF_i$,

  \item[\rm (iii)]
for  $1\leq i\leq N$,  $F_{i-1}B^*F_i$ is a weighted sum involving the following terms and coefficients:

\medskip
\centerline{
\begin{tabular}[t]{c|l}\label{fr1f}
{\rm term} & {\rm coefficient}
   \\  \hline
$LF_i$ & $\alpha^*_i(\theta_i-\theta_{i-1})\bigl(\theta^{*2}_{i}
-\beta\theta^*_{i}\theta^*_{i-1}+\theta^{*2}_{i-1}-\gamma^*(\theta^*_{i}+\theta^*_{i-1})
-\varrho^*\bigr)$
  \\
$L^2RF_i$ & $\alpha_i\alpha^*_{i+1}\alpha^*_i(\theta^*_{i-1}-\beta \theta^*_i
+\theta^*_{i+1}-\gamma^*)$
   \\
 $RL^2F_i$ &  $\alpha^*_i\alpha^*_{i-1}\alpha_{i-2}(-\theta^*_{i}+\beta \theta^*_{i-1}-\theta^*_{i-2}+\gamma^*)$
 \\
\end{tabular}}

\item[\rm (iv)]
for  $0\leq i\leq N$, $F_{i}B^*F_i$ is a weighted sum involving the following terms and  coefficients:

\centerline{
\begin{tabular}[t]{c|l}\label{fr0f}
{\rm term} & {\rm coefficient}
   \\  \hline
$RLF_i$ & $-\alpha^*_i\alpha_{i-1}\bigl(\theta^{*2}_{i-1}
-\beta\theta^*_{i-1}\theta^*_{i}+\theta^{*2}_{i}-\gamma^*(\theta^*_{i-1}+\theta^*_{i})
-\varrho^*\bigr)$
  \\
$LRF_i$ & $\alpha^*_{i+1}\alpha_i\bigl(\theta^{*2}_i
-\beta\theta^*_i\theta^*_{i+1}+\theta^{*2}_{i+1}-\gamma^*(\theta^*_i+\theta^*_{i+1})
-\varrho^*\bigr)$
 \\
\end{tabular}}

\item[\rm (v)]
for  $2\leq i\leq N$, $F_{i-2}B^*F_i$ is a weighted sum involving the following terms and coefficients:

\medskip
\centerline{
\begin{tabular}[t]{c|l}
{\rm term} & {\rm coefficient}
   \\  \hline
$L^2F_i$ & $\small{\alpha^*_i\alpha^*_{i-1}\Big((\theta_i-\theta_{i-2})(\theta^*_{i}
+\theta^*_{i-1}+\theta^*_{i-2}-\gamma^*)
-(\beta+1)\bigl((\theta_{i-1}-\theta_{i-2})\theta^*_{i-2}
+(\theta_i-\theta_{i-1})\theta^*_i\bigr)\Big)}$
  \\
$L^3RF_i$ & $\alpha^*_i\alpha^*_{i-1}\alpha^*_{i+1}\alpha_i$
   \\
 $L^2RLF_i$ & $-(\beta+1)\alpha^{*2}_i\alpha^*_{i-1}\alpha_{i-1}$
 \\
$LRL^2F_i$ & $(\beta+1)\alpha^*_i\alpha^{*2}_{i-1}\alpha_{i-2}$
\\
 $RL^3F_i$ & $-\alpha^*_i\alpha^*_{i-1}\alpha^*_{i-2}\alpha_{i-3}$
\end{tabular}}
\end{enumerate}

\end{lemma}
\begin{proof}
Combine  (\ref{12151}), (\ref{12153}), (\ref{12154}), (\ref{leonard1}), (\ref{leonard2}), (\ref{equa2sim}).
\end{proof}

At the beginning of Section 11, we defined some scalars $\{\theta_i\}^N_{i=0}$, $\{\theta^*_i\}^N_{i=0}$ and
$\beta, \gamma, \gamma^*, \varrho, \varrho^*$. As we proceed, we often make an assumption about these
scalars called the {\em standard assumption}.

\begin{definition}\label{def:0427}
{\rm Under the standard assumption,
\begin{enumerate}
\item[\rm(i)]
$
\beta+1=\frac{\theta_{i-2}-\theta_{i+1}}{\theta_{i-1}-\theta_i}
\qquad\qquad\qquad (2\leq i\leq N-1),
$
\item[\rm(ii)]
$
\beta+1
=\frac{\theta^*_{i-2}-\theta^*_{i+1}}{\theta^*_{i-1}-\theta^*_i}\qquad\qquad\qquad (2\leq i\leq N-1),
$

\item[\rm(iii)] $\gamma=\theta_{i-1}-\beta \theta_i+\theta_{i+1}\qquad\qquad\qquad (1\leq i\leq N-1),$

\item[\rm(iv)] $\gamma^*=\theta^*_{i-1}-\beta \theta^*_i+\theta^*_{i+1}\qquad\qquad\qquad (1\leq i\leq N-1),$

\item[\rm(v)] $\varrho=\theta^2_{i-1}-\beta \theta_{i-1}\theta_i+\theta^2_{i}-
\gamma (\theta_{i-1}+\theta_i)\qquad\qquad\qquad (1\leq i\leq N),$

\item[\rm(vi)] $\varrho^*=\theta^{*2}_{i-1}-\beta \theta^*_{i-1}\theta^*_i+\theta^{*2}_{i}
-\gamma^* (\theta^*_{i-1}+\theta^*_i)\qquad\qquad\qquad (1\leq i\leq N).$
\end{enumerate}}
\end{definition}

\begin{lemma}\label{lem:parasatisfy}
Assume $B=0$ and $B^*=0$. Then the standard assumption holds.
\end{lemma}
\begin{proof} We Refer to Definition~\ref{def:0427}.
To get (ii),  (v)
combine  Lemma~\ref{lem:fbf},  Lemma~\ref{lem:btri}(i), (ii) along with (\ref{alph0}), (\ref{alph}), (\ref{thetasatisfy}). Using (v), one gets (iii).
To get (i), (vi), combine  Lemma~\ref{lem:fbf}, Lemma~\ref{lem:bstartri}(i), (ii) along with (\ref{alph0}), (\ref{alph}), (\ref{thetasatisfy}).
 Using (vi), one gets (iv).
\end{proof}

For notational convenience  define
$$
\heartsuit_i=(\beta+1)\bigl((\theta^*_i-\theta^*_{i+2})(\theta_{i+1}-\theta_i)
+(\theta^*_{i+1}-\theta^*_{i+2})(\theta_i-\theta_{i+2})\bigr)\qquad\qquad (0\leq i\leq N-2).
$$
\begin{lemma}\label{lem:1602051}
Under the standard assumption, the following hold.
\begin{description}
  \item[\rm(i)] For $0\leq i\leq N-2$, the scalar $\heartsuit_i$ is equal to the expression
  \begin{equation}\label{201602101}
  (\theta^*_i-\theta^*_{i+2})(\theta_{i}
+\theta_{i+1}+\theta_{i+2}-\gamma)
-(\beta+1)\bigl((\theta^*_{i+1}-\theta^*_{i+2})\theta_{i+2}
+(\theta^*_i-\theta^*_{i+1})\theta_i\bigr)
\end{equation}
  that appears in the table from Lemma~\ref{lem:btri}{\rm(v)}.
  \item[\rm(ii)] For $2\leq i\leq N$, the scalar $-\heartsuit_{i-2}$ is equal to the expression
\begin{equation}\label{201602102}
  (\theta_i-\theta_{i-2})(\theta^*_{i}
+\theta^*_{i-1}+\theta^*_{i-2}-\gamma^*)
-(\beta+1)\bigl((\theta_{i-1}-\theta_{i-2})\theta^*_{i-2}
+(\theta_i-\theta_{i-1})\theta^*_i\bigr)
\end{equation}
that appears in the table from Lemma~\ref{lem:bstartri}{\rm(v)}.
\end{description}
\end{lemma}
\begin{proof}
(i) Evaluate (\ref{201602101}) using  Lemma~\ref{lem:parasatisfy}(i), (iii) and simplify the result.

\noindent
(ii) Evaluate (\ref{201602102}) using  Lemma~\ref{lem:parasatisfy}(ii), (iv) and simplify the result.
\end{proof}
\begin{lemma}\label{lem:1602021}
Under the standard assumption,     for $0\leq i\leq N-2$,
$F_{i+2}BF_i$ is a weighted sum  involving the following terms and coefficients:

\medskip
\centerline{
\begin{tabular}[t]{c|l}
{\rm term} & {\rm coefficient}
   \\  \hline
$R^2F_i$ & $-\alpha_{i}\alpha_{i+1}\bigl(q^{N+M-i-2}(q+1)(\beta+1)(\alpha_{i}\alpha^*_{i+1}-\alpha_{i+1}\alpha^*_{i+2})-\heartsuit_i\bigr)$
   \\
$R^3LF_i$ & $\alpha_{i}\alpha_{i+1}\bigl(\alpha_{i-1}\alpha^*_{i}-\frac{q(q^2-1)}{q^3-1}(\beta+1)\alpha_{i}\alpha^*_{i+1}+\frac{q^2(q-1)}{q^3-1}(\beta+1)\alpha_{i+1}\alpha^*_{i+2}\bigr)$
 \\
$LR^3F_i$ & $-\alpha_{i}\alpha_{i+1}\bigl(\frac{q-1}{q^3-1}(\beta+1)\alpha_{i}\alpha^*_{i+1}-\frac{q^2-1}{q^3-1}(\beta+1)\alpha_{i+1}\alpha^*_{i+2}+\alpha_{i+2}\alpha^*_{i+3}\bigr)$
 \end{tabular}}
 \medskip
 \noindent
 For $i=0$, the term $R^3LF_i=0$ and so the second row of the above table is eliminated.
 For $i=N-2$, the term $LR^3F_i=0$ and so the third row of the above table is eliminated.
\end{lemma}
\begin{proof}
 By (\ref{201603031}), (\ref{201603032}) we obtain
\begin{eqnarray}
& &R^2LRF_i=q^{N+M-i-2}(q+1)R^2F_i+\frac{q(q^2-1)R^3LF_i+(q-1)LR^3F_i}{q^3-1},\label{201603037}\\
& &RLR^2F_i=q^{N+M-i-2}(q+1)R^2F_i+\frac{q^2(q-1)R^3LF_i+(q^2-1)LR^3F_i}{q^3-1}.\label{201603038}
\end{eqnarray}
Consider the weighted sum description of $F_{i+2}BF_i$  in Lemma~\ref{lem:btri}(v).
Simplify this description by eliminating $R^2LRF_i$  using (\ref{201603037})
and  $RLR^2F_i$  using (\ref{201603038}). Combining this with Lemma~\ref{lem:1602051}(i),  we get the result.
\end{proof}
Referring to Lemma~\ref{lem:1602021}, for $i=0$ we obtain the following simplification.
\begin{lemma}\label{lem:1603032}
Under the standard assumption, the matrix $F_2BF_0$ is equal to $-\alpha_0\alpha_1q^MR^2F_0$ times the scalar
\begin{eqnarray*}\label{201603035}
\frac{q^N-1}{q-1}(\beta+1)\alpha_0\alpha^*_1-\frac{q^{N-1}-1}{q-1}~\frac{q^2-1}{q-1}(\beta+1)\alpha_1\alpha^*_2
+\frac{q^{N-2}-1}{q-1}~\frac{q^3-1}{q-1}\alpha_2\alpha^*_3-q^{-M}\heartsuit_0.
\end{eqnarray*}
\end{lemma}
\begin{proof}
In Lemma~\ref{lem:1602021}, set $i=0$ to obtain $F_2BF_0$ as a weighted sum. In this weighted sum, eliminate $LR^3F_0$ using (\ref{201602271})
and simplify the result.
\end{proof}

\begin{lemma}\label{lem:1602029}
Under the standard assumption,
 for $2\leq i\leq N$,
$F_{i-2}B^*F_i$ is a weighted sum  involving the following terms and coefficients:

\medskip
\centerline{
\begin{tabular}[t]{c|l}
{\rm term} & {\rm coefficient}
   \\  \hline
$L^2F_i$ & $-\alpha^*_{i-1}\alpha^*_i\bigl(q^{N+M-i}(q+1)(\beta+1)(\alpha_{i-1}\alpha^*_{i}-\alpha_{i-2}\alpha^*_{i-1})+\heartsuit_{i-2}\bigr)$
   \\
$RL^3F_i$ & $-\alpha^*_{i-1}\alpha^*_i\bigl(\alpha_{i-3}\alpha^*_{i-2}-\frac{q(q^2-1)}{q^3-1}(\beta+1)\alpha_{i-2}\alpha^*_{i-1}+\frac{q^2(q-1)}{q^3-1}(\beta+1)
\alpha_{i-1}\alpha^*_{i}\bigr)$
 \\
$L^3RF_i$ & $\alpha^*_{i-1}\alpha^*_i\bigl(\frac{q-1}{q^3-1}(\beta+1)\alpha_{i-2}\alpha^*_{i-1}-\frac{q^2-1}{q^3-1}(\beta+1)\alpha_{i-1}\alpha^*_{i}+
\alpha_{i}\alpha^*_{i+1}\bigr)$
 \end{tabular}}
 \medskip
 \noindent
 For $i=2$, the term $RL^3F_i=0$ and so the second row of the above table is eliminated.
 For $i=N$, the term $L^3RF_i=0$ and so the third row of the above table is eliminated.
\end{lemma}

\begin{proof}
Similar to the proof of Lemma~\ref{lem:1602021}.
\end{proof}

Referring to Lemma~\ref{lem:1602029}, for $i=2$ we obtain the following simplification.
\begin{lemma}\label{lem:1603033}
Under the standard assumption,
the matrix $F_0B^*F_2$ is equal to $\alpha^*_1\alpha^*_2q^ML^2F_2$ times
\begin{eqnarray*}\label{201603039}
\frac{q^N-1}{q-1}(\beta+1)\alpha_0\alpha^*_1-\frac{q^{N-1}-1}{q-1}~\frac{q^2-1}{q-1}(\beta+1)\alpha_1\alpha^*_2
+\frac{q^{N-2}-1}{q-1}~\frac{q^3-1}{q-1}\alpha_2\alpha^*_3-q^{-M}\heartsuit_0.\end{eqnarray*}
\end{lemma}
\begin{proof}
In Lemma~\ref{lem:1602029}, set $i=2$ to obtain $F_0B^*F_2$ as a weighted sum. In this weighted sum, eliminate  $L^3RF_2$ using (\ref{201602272}) and simplify the result.
\end{proof}

\begin{lemma}\label{lem:20160227ss}
Assume $B=0$ and $B^*=0$. Then  the scalar $q^{-M}\heartsuit_0$ is equal to
\begin{eqnarray*}
\frac{q^N-1}{q-1}(\beta+1)\alpha_0\alpha^*_1-\frac{q^{N-1}-1}{q-1}~\frac{q^2-1}{q-1}(\beta+1)\alpha_1\alpha^*_2
+\frac{q^{N-2}-1}{q-1}~\frac{q^3-1}{q-1}\alpha_2\alpha^*_3.
\end{eqnarray*}
\end{lemma}
\begin{proof}
Use Lemma~\ref{lem:fbf} with $D=B$, along with   Lemma~\ref{lem:1603032}.
\end{proof}

\begin{lemma}\label{lem:xidependent}
Assume $B=0$ and $B^*=0$. Then for  $1\leq i\leq N-2$,
\begin{eqnarray}
& &q^{N+M-i-2}(q+1)(\beta+1)(\xi_{i}-\xi_{i+1})-\heartsuit_i=0,\label{labelr2lr}\\
& &\xi_{i-1}-\frac{q(q^2-1)}{q^3-1}(\beta+1)\xi_{i}+\frac{q^2(q-1)}{q^3-1}(\beta+1)\xi_{i+1}=0. \label{labelrlr2}
\end{eqnarray}
Moreover for $1\leq i\leq N-3$,
\begin{eqnarray}\label{xsatisfy1}
\frac{q-1}{q^3-1}(\beta+1)\xi_{i}-\frac{q^2-1}{q^3-1}(\beta+1)\xi_{i+1}+\xi_{i+2}=0.
\end{eqnarray}
\end{lemma}
\begin{proof}

First assume that $1\leq i\leq N-3$. In the table of Lemma~\ref{lem:1602021}, each coefficient is zero by  Lemma~\ref{lem:fbf} and  Lemma~\ref{lem:rli0}(i).
This gives (\ref{labelr2lr})--(\ref{xsatisfy1}). Next assume that $i=N-2$. In the table of Lemma~\ref{lem:1602021}, the term $LR^3F_i=0$. So the third equation
of the table is eliminated. In the remaining two rows, each coefficient is zero by  Lemma~\ref{lem:fbf} and  Lemma~\ref{lem:rli311}(i).
This gives   (\ref{labelr2lr}), (\ref{labelrlr2}).
\end{proof}

Next we simplify  the equations (\ref{labelr2lr})--(\ref{xsatisfy1}).

\begin{lemma}\label{lem:xisatisfy}
Assume $B=0$ and $B^*=0$. Then for  $1\leq i\leq N-3$,
\begin{enumerate}
\item[\rm(i)]\label{xsatisfy4}
$\xi_{i-1}-(\beta+1)\xi_i+(\beta+1)\xi_{i+1}-\xi_{i+2}=0,$
\item[\rm(ii)]\label{xsatisfy2}
$q^{-1}\xi_{i-1}-(\beta+1)\xi_{i+1}+(q+1)\xi_{i+2}=0,$
\item[\rm(iii)]
$\xi_{i-1}-\xi_{i+2}-(q+1)^{-1}q^{2+i-N-M}\heartsuit_i=0.$
\end{enumerate}
\end{lemma}
\begin{proof}
(i) Combine   (\ref{labelrlr2}), (\ref{xsatisfy1}).

\noindent
(ii)
Combine   (\ref{labelrlr2}), (\ref{xsatisfy1}).

\noindent
(iii) By (\ref{labelr2lr}) and (i).
\end{proof}

\begin{lemma}\label{lem:threerecurren}
Assume $B=0$ and $B^*=0$. Then the sequences $\{\theta_i\}_{i=0}^N$,  $\{\theta^*_i\}_{i=0}^N$,  $\{\xi_i\}_{i=0}^{N-1}$
are all $\beta$-recurrent.
\end{lemma}
\begin{proof}
By Lemma~\ref{lem:parasatisfy}, $\{\theta_i\}_{i=0}^N$ and  $\{\theta^*_i\}_{i=0}^N$ are $\beta$-recurrent.  By Lemma~\ref{lem:xisatisfy}(i),
 $\{\xi_i\}_{i=0}^{N-1}$ is $\beta$-recurrent.
\end{proof}

\begin{lemma}\label{lem:betanot2}
Assume  $B=0$ and $B^*=0$. Then $\beta\neq2$ and $\beta\neq-2$.
\end{lemma}
\begin{proof}
First we show $\beta\neq2$. Suppose on the contrary that $\beta=2$.
 Then by Lemma~\ref{lem:xisatisfy}(ii) we find that for  $1\leq i\leq N-3$,
\begin{equation}\label{beta22}
q^{-1}\xi_{i-1}-3\xi_{i+1}+(q+1)\xi_{i+2}=0.
\end{equation}
By Lemma~\ref{lem:betarecurrence}(i) there exist scalars $a,b,c\in \mathbb{C}$ such that
 $\xi_i=a+bi+ci^2$ for $0\leq i\leq N-1$. Evaluating  (\ref{beta22}) using this, we find that  for  $1\leq i\leq N-3$,
 $0$ is a weighted sum involving the following terms and coefficients:

 \medskip
\centerline{
\begin{tabular}[t]{c|l}
term & coefficient
   \\  \hline
$i^2$ & $(q-1)^2c$
  \\
$i$ & $b-2c-3q(b+2c)+q(q+1)(b+4c)$
   \\
 $1$ &  $a-b+c-3q(a+b+c)+q(q+1)(a+2b+4c)$
 \\
     \end{tabular}}
 \medskip
\noindent
Since $N\geq 6$, in the above table each coefficient is $0$. By construction, $q$ is an integer at least $2$.
Then $c=0$ by the first row of the table.
Combining  this with the second row of the table
we get $(q-1)^2b=0$, so $b=0$. Combining $b=0$ and $c=0$ with the third row of the table we get $(q-1)^2a=0$, so $a=0$.
Hence $\xi_i=0$ for $0\leq i\leq N-1$. This  contradicts  the line below (\ref{xala}),
so $\beta\neq 2$.

Now we show $\beta\neq -2$. Suppose on the contrary that $\beta=-2$.
 Then by Lemma~\ref{lem:xisatisfy}(ii) we find that for  $1\leq i\leq N-3$,
\begin{equation}\label{beta2203}
q^{-1}\xi_{i-1}+\xi_{i+1}+(q+1)\xi_{i+2}=0.
\end{equation}
By Lemma~\ref{lem:betarecurrence}(ii) there exist scalars $a,b,c\in \mathbb{C}$ such that
 $\xi_i=a+b(-1)^i+ci(-1)^i$ for $0\leq i\leq N-1$. Evaluating  (\ref{beta2203}) using this, we find that  for  $1\leq i\leq N-3$,
 $0$ is a weighted sum involving the following terms and coefficients:

 \medskip
\centerline{
\begin{tabular}[t]{c|l}
term & coefficient
   \\  \hline
$i(-1)^{i}$ & $(q^2-1)c$
  \\
$(-1)^{i}$ & $(q^2-1)b+(1+q+2q^2)c$
   \\
 $1$ &  $(1+q)^2a$
 \\
     \end{tabular}}
 \medskip
\noindent
 Since $N\geq 6$, in the above table each coefficient is $0$. By construction, $q$ is an integer at least $2$.
 Then $c=0$ by the first row of the table and $a=0$ by the third row of the table. Combining $c=0$
 with the second row of the table we get $(1-q^2)b=0$, so $b=0$.
Hence $\xi_i=0$ for $0\leq i\leq N-1$. This  contradicts  the line below (\ref{xala}),
so $\beta\neq -2$.
\end{proof}

\begin{definition}\label{def:bigq}
Fix a nonzero $Q\in \mathbb{C}$ such that $\beta=Q+Q^{-1}$.
\end{definition}
Assume $B=0$ and $B^*=0$. Then the scalar $Q$ in Definition~\ref{def:bigq} is not equal to $\pm 1$ by Lemma~\ref{lem:betanot2}.
\begin{lemma}\label{lem:xiforms}
Assume $B=0$ and $B^*=0$. Then there exist scalars $a,b,c, a^*,b^*,c^*$ in $\mathbb{C}$
such that for $0\leq i\leq N$,
\begin{eqnarray}
\theta_i&=&a+bQ^i+cQ^{-i},\label{theta}
\\
\theta^*_i&=&a^*+b^*Q^i+c^*Q^{-i}.\label{theta*}
\end{eqnarray}
Also, there exist scalars $x,y,z$ in $\mathbb{C}$
such that for $0\leq i\leq N-1$,
\begin{eqnarray}
\xi_i&=&x+yQ^i+zQ^{-i}.\label{X}
\end{eqnarray}
The scalar $Q$ is from Definition~\ref{def:bigq}.
\end{lemma}
\begin{proof}
Combine  Lemmas~\ref{lem:betarecurrence}, \ref{lem:threerecurren}, \ref{lem:betanot2}, Definition~\ref{def:bigq}.
\end{proof}

\begin{note}\label{note:2}\rm
Assume $B=0$ and $B^*=0$. Then referring to Lemma~\ref{lem:xiforms}, the scalars $b$, $c$ are not both zero by (\ref{thetasatisfy}). Similarly, the scalars $b^*$, $c^*$ are not both zero.
\end{note}
\begin{note}\label{note:1}\rm
Assume $B=0$ and $B^*=0$. Then referring to  Lemma~\ref{lem:xiforms}, by (\ref{thetasatisfy})  we have  that
$Q^i\neq 1$ for $1\leq i\leq N$.
\end{note}

\begin{lemma}\label{thm:main41}
Assume $B=0$ and $B^*=0$. Then the scalars $\beta, \gamma, \gamma^*, \varrho, \varrho^*$ from above {\rm(\ref{equa1})}  are given by
$\beta=Q+Q^{-1}$ and
\begin{eqnarray*}
& &\gamma=-Q^{-1}(Q-1)^2a,\qquad \qquad\qquad\varrho =Q^{-1}(Q-1)^2a^2-(Q-Q^{-1})^2bc,\\
& &\gamma^*=-Q^{-1}(Q-1)^2a^*,\qquad \qquad\quad~
\varrho^*=Q^{-1}(Q-1)^2a^{*2}-(Q-Q^{-1})^2b^*c^*.
\end{eqnarray*}
The scalar $Q$ is from Definition~\ref{def:bigq} and the scalars $a, b, c, a^*, b^*, c^*$ are from Lemma~\ref{lem:xiforms}.
\end{lemma}
\begin{proof}
Use   Lemma~\ref{lem:parasatisfy}  and (\ref{theta}), (\ref{theta*}).
\end{proof}

\begin{lemma}\label{lem:1602052}
Assume $B=0$ and $B^*=0$. Then $y=0$ or $z=0$.
\end{lemma}
\begin{proof}
Evaluate Lemma~\ref{lem:xisatisfy}(ii) using (\ref{X}).
Then for
$1\leq i\leq N-3$,  $0$ is equal to a weighted sum involving the following terms and coefficients:

\medskip
\centerline{
\begin{tabular}[t]{c|l}
term & coefficient
   \\  \hline
$1$ & $q^{-1}(1-qQ)(1-qQ^{-1})x$
  \\
$Q^i$ & $(q^{-1}Q^{-1}-Q)(1-qQ)y$
   \\
 $Q^{-i}$ &  $(q^{-1}Q-Q^{-1})(1-qQ^{-1})z$
 \\
     \end{tabular}}
 \medskip
\noindent

Recall that $Q\neq\pm 1$ and $N\geq 6$. Therefore, in the above table each coefficient is $0$. In other words,
\begin{eqnarray}
& &(1-qQ)(1-qQ^{-1})x=0,\label{201601065}\\
& &(q^{-1}Q^{-1}-Q)(1-qQ)y=0,\label{201601068}\\
& &(q^{-1}Q-Q^{-1})(1-qQ^{-1})z=0.\label{201601060}
\end{eqnarray}
We now show that $yz=0$. Suppose not. Then both
\begin{eqnarray}
& &(q^{-1}Q^{-1}-Q)(1-qQ)=0,\label{201602051}\\
& &(q^{-1}Q-Q^{-1})(1-qQ^{-1})=0.\label{201602052}
\end{eqnarray}
By (\ref{201602051}), one gets $Q=q^{-1}$ or $Q^2=q^{-1}$. By (\ref{201602052}), one gets $Q=q$ or $Q^{2}=q$.
This contradicts
 Note~\ref{note:1}.  Therefore $yz=0$ and the  result follows.
\end{proof}

We now investigate the cases in Lemma~\ref{lem:1602052}. In what follows, we refer to the scalars in Definition~\ref{def:bigq} and Lemma~\ref{lem:xiforms}.

\begin{lemma}\label{lem:Qrelationq}
Assume $B=0$ and $B^*=0$.
\begin{description}
  \item[\rm Case~I:] Assume $y=0$ and $z\neq 0$. Then
  $$Q=q,\qquad bb^*=0,\qquad z=-q^{-1-N-M}(q-1)^2cc^*, \qquad c\neq0,\qquad c^*\neq0.$$

  \item[\rm Case~II:] Assume $y\neq 0$ and $z=0$. Then
  $$Q=q^{-1},\qquad cc^*=0,\qquad y=-q^{-1-N-M}(q-1)^2bb^*,\qquad b\neq0,\qquad b^*\neq0.$$

 \item[\rm Case~III:] Assume $y=0$ and $z=0$. Then
    $$Q\in \{q, q^{-1}\},\qquad x\neq0, \qquad bb^*=0,\qquad cc^*=0.$$
\end{description}
\end{lemma}
\begin{proof}
Evaluate Lemma~\ref{lem:xisatisfy}(iii) using (\ref{theta}), (\ref{theta*}), (\ref{X}).
Then for
$1\leq i\leq N-3$,  $0$ is equal to a weighted sum involving the following terms and coefficients:

\medskip
\centerline{
\begin{tabular}[t]{c|l}
term & coefficient
   \\  \hline
$Q^i$ & $Q^{-1}(1-Q^3)y$
   \\
 $Q^{-i}$ &  $Q(1-Q^{-3})z$
 \\
 $(qQ^2)^{i}$ &  $-(q+1)^{-1}q^{2-N-M}Q^{-1}(Q^3-1)(Q^2-1)(Q-1)bb^*$
 \\
 $(qQ^{-2})^{i}$ &  $-(q+1)^{-1}q^{2-N-M}Q(Q^{-3}-1)(Q^{-2}-1)(Q^{-1}-1)cc^*$
 \\
     \end{tabular}
}

\medskip
We now consider the cases.

Case~I: Using
(\ref{201601060}) one gets $Q=q$ or $Q^{2}=q$.
We claim that $Q=q$.  Otherwise, we have $qQ^2=Q^4$ and $qQ^{-2}=1$. Then by Note~\ref{note:1}, the scalars $Q^{-1}, qQ^2, qQ^{-2}$ are mutually distinct.
So in the above table the coefficient of each term is zero. In particular,
$Q(1-Q^{-3})z=0$. By Note~\ref{note:1} again, one gets $z=0$, a contradiction. We have shown that $Q=q$. Using this fact, we simplify the above table, and
find that
 $0$ is equal to a weighted sum involving the following terms and coefficients:

\centerline{
\begin{tabular}[t]{c|l}
term & coefficient
   \\  \hline
 $q^{-i}$ &  $q(1-q^{-3})z-(q+1)^{-1}q^{3-N-M}(q^{-3}-1)(q^{-2}-1)(q^{-1}-1)cc^*$
 \\
 $q^{3i}$ &  $-(q+1)^{-1}q^{1-N-M}(q^3-1)(q^2-1)(q-1)bb^*$
 \\
     \end{tabular}
}

\medskip
\noindent
In this table the coefficient of each term is zero. Therefore,
\begin{eqnarray}
& &q(1-q^{-3})z-(q+1)^{-1}q^{3-N-M}(q^{-3}-1)(q^{-2}-1)(q^{-1}-1)cc^*=0,\label{201601305}\\
& &-(q+1)^{-1}q^{1-N-M}(q^3-1)(q^2-1)(q-1)^2bb^*=0.\label{201601308}
\end{eqnarray}
Using (\ref{201601308}), we obtain $bb^*=0$. Using (\ref{201601305}), we obtain
$z=-q^{-1-N-M}(q-1)^2cc^*$.

Case~II: Similar to the proof of Case~I.

Case III:
By (\ref{X}) and since $\xi_i\neq 0$ for $0\leq i\leq N-1$, we find  $x\neq 0$.  Combine this with  (\ref{201601065}) to get
$Q=q$ or $Q=q^{-1}$. Now $Q^2\neq Q^{-2}$. Moreover, for $1\leq i\leq N-3$,
 $0$ is equal to a weighted sum involving the following terms and coefficients:

\centerline{
\begin{tabular}[t]{c|l}
term & coefficient
   \\  \hline
 $(qQ^2)^{i}$ &  $-(q+1)^{-1}q^{2-N-M}Q(Q^{-3}-1)(Q^{-2}-1)(Q^{-1}-1)cc^*$
 \\
 $(qQ^{-2})^{i}$ &  $-(q+1)^{-1}q^{2-N-M}Q^{-1}(Q^3-1)(Q^2-1)(Q-1)bb^*$
 \\
     \end{tabular}
}
\medskip
\noindent
In this table the coefficient of each term  is zero.
Therefore $bb^*=0$ and $cc^*=0$.
\end{proof}

\begin{note}\label{note:0427}{\rm
Referring to Definition~\ref{def:bigq}, the scalar $Q$ is defined up to inverse. In Case III, after replacing
$Q$ by $Q^{-1}$ if necessary, we will assume $Q=q$.}
\end{note}

\begin{definition}\label{def:1602051}{\rm
Consider  Cases  I, II, III in Lemma~\ref{lem:Qrelationq}.

We  partition Case I into subcases:
\begin{description}

  \item[{\rm Case I$^+$:}] \qquad\qquad\qquad\qquad\qquad$b\neq0,\qquad\qquad\qquad\quad b^*=0;$

  \item[{\rm Case I$^-$:}] \qquad\qquad\qquad\qquad\qquad$b=0,\qquad\qquad\qquad\quad b^*\neq0;$

   \item[{\rm Case I$^0$:}] \qquad\qquad\qquad\qquad\qquad~$b=0,\qquad\qquad\qquad\quad b^*=0.$
  \end{description}

  We  partition Case II into subcases:
\begin{description}

   \item[{\rm Case II$^+$:}] \qquad\qquad\qquad\qquad\quad~~$c\neq 0,\qquad\qquad\quad\qquad c^*=0;$

    \item[{\rm Case II$^-$:}] \qquad\qquad\qquad\qquad\quad~~$c=0,\qquad\qquad\qquad\quad c^*\neq0;$

     \item[{\rm Case II$^0$:}] \qquad\qquad\qquad\qquad\qquad$c=0,\qquad\qquad\qquad\quad c^*=0.$
      \end{description}

  We  partition Case III into subcases:
\begin{description}

     \item[{\rm Case III$^+$:}] \qquad\quad$b\neq0$, \qquad $b^*=0$, \qquad\qquad $c=0$,\qquad $c^*\neq0$, \qquad\qquad $Q=q$;

    \item[{\rm Case III$^-$:}] \qquad\quad$b=0$, \qquad $b^*\neq0$, \qquad\qquad $c\neq0$,\qquad $c^*=0$, \qquad\qquad $Q=q$.

   \end{description}}

\end{definition}

\begin{theorem}\label{thm:main3}
Assume $B=0$ and $B^*=0$. Then $\{\xi_i\}_{i=0}^{N-1}, \{\theta_i\}_{i=0}^N, \{\theta^*_i\}_{i=0}^N$ are given in the table below:

\medskip
\centerline{
\begin{tabular}[t]{l|c c c}
{\rm Case} & $\xi_i$ & $\theta_i$ & $\theta^*_i$
   \\  \hline\hline
${\rm I^+}$ & $x-q^{-1-N-M-i}(q-1)^2cc^*$ & $a+bq^i+cq^{-i}$ & $a^*+c^*q^{-i}$
 \\
${\rm I^-}$ & $x-q^{-1-N-M-i}(q-1)^2cc^*$ & $a+cq^{-i}$ & $a^*+b^*q^i+c^*q^{-i}$
 \\
${\rm I^0}$ & $x-q^{-1-N-M-i}(q-1)^2cc^*$ & $a+cq^{-i}$ & $a^*+c^*q^{-i}$
 \\
  \hline
${\rm II^+}$ & $x-q^{-1-N-M-i}(q-1)^2bb^*$ & $a+bq^{-i}+cq^{i}$ & $a^*+b^*q^{-i}$
 \\
${\rm II^-}$ & $x-q^{-1-N-M-i}(q-1)^2bb^*$ & $a+bq^{-i}$ & $a^*+b^*q^{-i}+c^*q^{i}$
\\
${\rm II^0}$ & $x-q^{-1-N-M-i}(q-1)^2bb^*$ & $a+bq^{-i}$ & $a^*+b^*q^{-i}$
\\
 \hline
${\rm III^+}$ & $x$ & $a+bq^i$ & $a^*+c^*q^{-i}$
\\
${\rm III^-}$ & $x$ & $a+cq^{-i}$ & $a^*+b^*q^{i}$
\end{tabular}}
\end{theorem}

\begin{proof}
By Lemmas~\ref{lem:xiforms}, \ref{lem:Qrelationq}, Definition~\ref{def:1602051}, and (\ref{thetasatisfy}).
\end{proof}

We are done assuming $B=0$ and $B^*=0$. We now go in the other logical direction.

\begin{theorem}\label{thm:main4}
Assume  $\{\xi_i\}_{i=0}^{N-1}$, $\{\theta_i\}_{i=0}^N, \{\theta^*_i\}_{i=0}^N$ satisfy  one of the cases  {\rm I$^+$, I$^-$, I$^0$,
II$^+$, II$^-$, II$^0$, III$^+$, III$^-$}  from  Theorem~\ref{thm:main3}. Define $\beta, \gamma, \gamma^*, \varrho, \varrho^*$ using
Lemma~\ref{lem:parasatisfy}.
Then referring to {\rm(\ref{equa1}), (\ref{equa2})} we have  $B=0$ and $B^*=0$.
\end{theorem}
\begin{proof}
For $0\leq i,j\leq N$ the matrix $F_jBF_i=0$ for the following reason:

\medskip
\centerline{
\begin{tabular}[t]{l| l}
case & reason
   \\  \hline\hline
 $i-j>1$ & Lemma~\ref{lem:1602241} \\
 $i-j=1$ & Lemma~\ref{lem:btri}(ii)\\
 $i-j=0$&  Lemma~\ref{lem:btri}(iv)\\
 $i-j=-1$&  Lemma~\ref{lem:btri}(iii)\\
 $i-j=-2$ and $i=0$ &  Lemma~\ref{lem:1603032}\\
 $i-j=-2$ and $1\leq i\leq N-2$& Lemma~\ref{lem:1602021}\\
 $i-j=-3$ & Lemma~\ref{lem:btri}(i)\\
 $i-j<-3$& Lemma~\ref{lem:1602241}
   \\
 \end{tabular}}
\medskip
\noindent
Now $B=0$ by Lemma~\ref{lem:fbf}.

For $0\leq i,j\leq N$ the matrix $F_jB^*F_i=0$ for the following reason:

\medskip
\centerline{
\begin{tabular}[t]{l| l}
case & reason
   \\  \hline\hline
 $i-j<-1$ & Lemma~\ref{lem:1602241} \\
 $i-j=-1$ & Lemma~\ref{lem:bstartri}(ii)\\
 $i-j=0$&  Lemma~\ref{lem:bstartri}(iv)\\
 $i-j=1$&  Lemma~\ref{lem:bstartri}(iii)\\
 $i-j=2$ and $i=2$ &  Lemma~\ref{lem:1603033}\\
 $i-j=2$ and $3\leq i\leq N$& Lemma~\ref{lem:1602029}\\
 $i-j=3$ & Lemma~\ref{lem:bstartri}(i)\\
 $i-j>3$& Lemma~\ref{lem:1602241}
   \\
 \end{tabular}}
\medskip
\noindent
Now $B^*=0$ by Lemma~\ref{lem:fbf}.
\end{proof}


\section{Leonard pairs based on $\mathcal{A}_q(N, M)$}
In this section we obtain some
 Leonard pairs from $\mathcal{A}_q(N, M)$. Define $A, A^*$  as in (\ref{leonard1}), (\ref{leonard2}). Assume $\{\xi_i\}_{i=0}^{N-1}$, $\{\theta_i\}_{i=0}^N, \{\theta^*_i\}_{i=0}^N$
are from  Theorem~\ref{thm:main3}.
Let $W$ denote an irreducible $T$-module with endpoint $r$ and diameter $d$.

\begin{theorem}\label{thm:main6}
 The elements  $A, A^*$ act on $W$ as a Leonard pair if and only if
\begin{eqnarray}\label{1602161}
x\neq q^{r+d-N-M}(q-1)^2(bc^*+cb^*)q^{-i}\qquad\qquad (1\leq i\leq d).
\end{eqnarray}
\end{theorem}
\begin{proof}
Recall the basis $\{w_i\}^d_{i=0}$ for $W$ in Lemma~\ref{lem:irredbasis}. With respect to this basis the matrix
representing $A$  is
$$
A:\qquad\left(\begin{array}{cccccc}
\theta_r&&&&&0\\
\alpha_r&\theta_{r+1}&&&&\\
&\alpha_{r+1}&\theta_{r+2}&&&\\
&&\cdots&\cdots&&\\
&&&\cdots&\cdots&\\
0&&&&\alpha_{r+d-1}&\theta_{r+d}
\end{array}\right),
$$
and  the matrix
representing  $A^*$  is
$$
A^*:\qquad\left(\begin{array}{cccccc}
\theta^*_r&\alpha^*_{r+1}x_{r+1}(r,d)&&&&0\\
&\theta^*_{r+1}&\alpha^*_{r+2}x_{r+2}(r,d)&&&\\
&&\theta^*_{r+2}&\cdots&&\\
&&&\cdots&\cdots&\\
&&&&\cdots&\alpha^*_{r+d}x_{r+d}(r,d)\\
0&&&&&\theta^*_{r+d}
\end{array}\right).
$$
Referring to Definition~\ref{def:notation}, set
$$
\varphi_i=\xi_{r+i-1}x_{r+i}(r,d)\qquad\qquad\qquad(1\leq i\leq d).
$$
Then $\varphi_i\neq 0$  for $1\leq i\leq d$ by Note~\ref{note:3} and (\ref{alph0}), (\ref{alph}).
Also, set
$$
\phi_i=\varphi_1\sum\limits_{h=0}^{i-1}\frac{\theta_{r+h}-\theta_{r+d-h}}{\theta_r-\theta_{r+d}}
+(\theta^*_{r+i}-\theta^*_r)(\theta_{r+d-i+1}-\theta_r)\qquad\qquad(1\leq i\leq d).
$$
Referring  to the cases from Theorem~\ref{thm:main3}, we have

\medskip
\centerline{
\begin{tabular}[t]{l|l}
{\rm Case} & $\qquad\qquad\qquad\qquad\phi_i$
   \\  \hline\hline
${\rm I^+}$ & $\Bigl(\frac{q^{N+M-r-d}}{(q-1)^2}x-bc^*q^{-i}\Bigr)(q^i-1)(q^{d-i+1}-1)$
 \\
${\rm I^-}$ & $\Bigl(\frac{q^{N+M-r-d}}{(q-1)^2}x-cb^*q^{i-d-1}\Bigr)(q^i-1)(q^{d-i+1}-1)$
 \\
${\rm I^0}$ & $q^{N+M-r-d}(q-1)^{-2}(q^i-1)(q^{d-i+1}-1)x$
 \\\hline
${\rm II^+}$ & $\Bigl(\frac{q^{N+M-r-d}}{(q-1)^2}x-cb^*q^{-i}\Bigr)(q^i-1)(q^{d-i+1}-1)$
 \\
${\rm II^-}$ & $\Bigl(\frac{q^{N+M-r-d}}{(q-1)^2}x-bc^*q^{i-d-1}\Bigr)(q^i-1)(q^{d-i+1}-1)$
\\
${\rm II^0}$ & $q^{N+M-r-d}(q-1)^{-2}(q^i-1)(q^{d-i+1}-1)x$
\\\hline
${\rm III^+}$ & $\Bigl(\frac{q^{N+M-r-d}}{(q-1)^2}x-bc^*q^{-i}\Bigr)(q^i-1)(q^{d-i+1}-1)$
\\
${\rm III^-}$ & $\Bigl(\frac{q^{N+M-r-d}}{(q-1)^2}x-cb^*q^{i-d-1}\Bigr)(q^i-1)(q^{d-i+1}-1)$
     \end{tabular}}
\noindent
Using (\ref{1602161}) we find that in each case,  $\phi_i\neq 0$  for $1\leq i\leq d$.

In each case,
it is routine to check using Definition~\ref{def:nonzero} that
$(\{\theta_i\}_{i=0}^d, \{\theta^*_i\}_{i=0}^d,
 \{\varphi_j\}^d_{j=1}, \{\phi_j\}^d_{j=1})$
is a parameter array. By Lemma~\ref{thm:leoard equivalent}, the pair $A, A^*$ acts on $W$ as a Leonard pair.
\end{proof}

\section{Directions for further research}
In this section we mention some open problems.

\begin{problem}{\rm
Describe  $\Phi, \Omega$ in terms of $C_1, C_2$.}
\end{problem}

\begin{problem}{\rm
Describe  $\Phi, \Omega$ in terms of $R, L, K, K^{-1}$.}
\end{problem}

\begin{problem}{\rm
Find a combinatorial interpretation of  $\Phi, \Omega$.}
\end{problem}




\section{Acknowledgments}
The paper was written while the author was an honorary fellow at the University of Wisconsin-Madison, July 2015--July 2016. The  author would like to give her sincere thanks to
 professor Paul Terwilliger
for his guidance and   valuable
ideas.
The author is partially supported by NSF of China (11301138, 11271004),  NSF of Hebei Province (A2014205105)
and  Research Fund for the Doctoral Program of Hebei Normal University (L2015B02).

\noindent Wen Liu \hfil\break
\noindent College of Mathematics and Information \hfil\break
\noindent Hebei Normal University \hfil\break
\noindent No. 20 Road East of 2nd Ring South \hfil\break
\noindent Shijiazhuang Hebei, 050024, China \hfil\break
\noindent email: {\tt liuwen1975@126.com }\hfil\break

\end{document}